\documentclass[10pt]{amsart}

\usepackage[utf8]{inputenc}
\usepackage[T1]{fontenc}	
\usepackage[french,english]{babel}	

\usepackage{lmodern}			

\usepackage{graphicx}	

\usepackage[a4paper,plainpages=false,colorlinks,linkcolor=bleuFonce,citecolor=rougeFonce,urlcolor=vertFonce,breaklinks]{hyperref}

\usepackage{placeins}
\usepackage{float} 

\usepackage{tikz}
\usetikzlibrary{patterns}

\usepackage[normalem]{ulem}

\usepackage{color}
\definecolor{vertFonce}	{rgb}{0,0.5,0}
\definecolor{numLignes}	{rgb}{0.17,0.57,0.7}	
\definecolor{gris}		{rgb}{0.5,0.5,0.5}
\definecolor{grisFonce}	{rgb}{0.2,0.2,0.2}
\definecolor{orange}	{rgb}{1,0.65,0.31}		
\definecolor{orangeFonce}{rgb}{1,0.4,0}
\definecolor{bleuFonce}	{rgb}{0,0,0.4}
\definecolor{rougeFonce}{rgb}{0.3,0,0}
\definecolor{rougeWord}	{rgb}{0.5,0,0}
\definecolor{vertClair}	{rgb}{0.8,1,0.8}
\definecolor{rougeClair}{rgb}{1,0.5,0.5}
\definecolor{violet}	{rgb}{0.5,0,0.5}

\usepackage{pict2e}
\usepackage{multido}

\usepackage{amsfonts,amssymb,amsthm,amsmath} 
\usepackage{dsfont}				
\usepackage{mathrsfs}
\usepackage{yfonts}
\usepackage[mathscr]{euscript}
\usepackage{cancel}
\usepackage{enumerate}

\usepackage{mathtools} 
 
\newtheorem{lem}{Lemma}[section]
\newtheorem{theorem}{Theorem}[section]

\newtheorem{prop}{Proposition}[section]

\newtheorem{remark}{Remark}[section]

\newcommand{\step}[1]	{\paragraph{\itshape\bfseries #1.}}

\usepackage{stmaryrd}
\SetSymbolFont{stmry}{bold}{U}{stmry}{m}{n}
\newcommand		{\subsetArrow}	{\mathrel{\ooalign{$\subset$\cr%
\hidewidth\raise-.087ex\hbox{$_\shortrightarrow\mkern-1.5mu$}\cr}}}
\newcommand		{\subsetarrow}	{\mathrel{\ooalign{$\subset$\cr%
\hidewidth\raise-1.45ex\hbox{$\vec{}\mkern6mu$}\cr}}}

\newcommand		{\N}		{\mathbb N}			
\newcommand		{\RR}		{\mathbb R}			
\newcommand		{\R}		{\RR}
\newcommand		{\Z}		{\mathbb Z}
\newcommand		{\Rd}		{\R^3}
\newcommand		{\RRd}		{\R^6}
\newcommand		{\Rdd}		{\R^6}
\renewcommand	{\L}		{\mathcal L}		
\newcommand		{\cW}		{\mathcal W}		

\newcommand		{\cL}		{\mathcal L}		

\newcommand		\sfm 		{\mathsf m}

\newcommand     {\ecF}      {\mathscr{F}}

\newcommand		{\lt}			{\left}				%
\newcommand		{\rt}			{\right}			%
\renewcommand	{\(}			{\lt(}
\renewcommand	{\)}			{\rt)}
\newcommand		{\bangle}[1]	{\lt\langle #1\rt\rangle}
\newcommand		{\weight}[1]	{\bangle{#1}}	

\newcommand		{\com}[1]		{\lt[{#1}\rt]}		
\newcommand		{\bcom}[1]		{\big[{#1}\big]}

\newcommand		{\n}[1]			{\lt\lvert #1 \rt\rvert}
		
\newcommand		{\nrm}[1]		{\lt\lVert #1\rt\rVert}
\newcommand		{\bnrm}[1]		{\big\lVert #1\big\rVert}
\newcommand		{\Bnrm}[1]		{\Big\lVert #1\Big\rVert}
\newcommand		{\Nrm}[2]		{\nrm{#1}_{#2}}
\newcommand		{\bNrm}[2]		{\bnrm{#1}_{#2}}
\newcommand		{\BNrm}[2]		{\Bnrm{#1}_{#2}}

\renewcommand		{\d}		{\mathrm{d}}		
\newcommand			{\dd}		{\,\d}				
\newcommand			{\bd}		{\partial}			
\newcommand			{\dpt}		{\partial_t}

\newcommand			{\dt}		{\frac{\d}{\d t}}	

\newcommand			{\grad}		{\nabla}
\newcommand			{\lapl}		{\Delta}
\newcommand			{\Dx}		{\nabla_x}
\newcommand			{\Dv}		{\nabla_\xi}
\newcommand			{\Id}		{\mathrm{Id}}		

\DeclareMathOperator{\cF}		{\mathcal{F}}		
\DeclareMathOperator{\tr}		{Tr}				
\DeclareMathOperator{\diag}		{diag}

\newcommand		{\F}[1]			{\cF\!\( #1 \)}		
\newcommand		{\Tr}[1]		{\tr\!\( #1 \)}		
\newcommand		{\Diag}[1]		{\diag\!\( #1 \)}

\newcommand		{\intd}			{\int_{\Rd}}
\newcommand		{\intdd}		{\int_{\RRd}}
\newcommand		{\iintd}		{\iint_{\RRd}}
\newcommand		{\iintdd}		{\iint_{\RRd\times\RRd}}

\newcommand		{\ii}			{\mathrm{i}}	
\newcommand		{\init}			{\mathrm{in}}

\newcommand		{\eps}			{\varepsilon}

\newcommand		{\cC}			{\mathcal{C}}

\usepackage{braket}
\newcommand		{\Inprod}[2]	{\Braket{#1 | #2}}

\newcommand		{\op}		{\boldsymbol{\rho}}	
\newcommand{\bp}{\boldsymbol{p}}

\newcommand		{\opm}		{\sfm}				

\newcommand		{\opmu}		{\boldsymbol{\mu}}	

\newcommand		{\opnu}		{\boldsymbol{\nu}}	

\newcommand		{\opp}		{\boldsymbol{p}}

\newcommand		{\Dh}		{\boldsymbol{\nabla}}	
\newcommand		{\Dhx}[1]	{\Dh_{\!x} #1}			
\newcommand		{\Dhv}[1]	{\Dh_{\!\xi} #1}		

\title[On the $L^2$ rate of convergence]{On the $L^2$ Rate of Convergence in the Limit from the Hartree to the Vlasov--Poisson Equation}

\author[J. Chong]{Jacky J. Chong}
\author[L. Lafleche]{Laurent Lafleche}
\address[J. Chong, L. Lafleche]{Department of Mathematics, The University of Texas at Austin, Austin, TX 78712, USA, {\tt jwchong@math.utexas.edu}, {\tt lafleche@math.utexas.edu}}

\author[C. Saffirio]{Chiara Saffirio}
\address[C. Saffirio]{Department of Mathematics and Computer Science, University of Basel, 4051 Basel, Switzerland, {\tt chiara.saffirio@unibas.ch}}

\subjclass[2010]{81Q20, 35Q55, 35Q83 (82C10, 82C05).}
\keywords{semiclassical limit, Hartree equation, Vlasov equation, Coulomb potential, gravitational potential.}

\begin{document}
\begin{abstract}
	Using a new stability estimate for the difference of the square roots of two solutions of the Vlasov--Poisson equation, we obtain the convergence in the $L^2$ norm of the Wigner transform of a solution of the Hartree equation with Coulomb potential to a solution of the Vlasov--Poisson equation, with a rate of convergence proportional to $\hbar$. This improves  the $\hbar^{3/4-\varepsilon}$ rate of convergence in $L^2$ obtained in [L.~Lafleche, C.~Saffirio: \textit{Analysis \& PDE}, to appear]. Another reason of interest of this paper is the new method, reminiscent of the ones used to prove the mean-field limit from the many-body Schr\"odinger equation towards the Hartree--Fock equation for mixed states.
\end{abstract}

\maketitle

\bigskip

\renewcommand{\contentsname}{\centerline{Table of Contents}}
\setcounter{tocdepth}{2}	
\tableofcontents

\section{Introduction and main result}

	In this paper, we study the limit as $\hbar\to 0$ of the Hartree equation towards the Vlasov--Poisson equation. We focus on the cases of the three dimensional Coulomb and gravitational interaction potentials 
	\begin{equation}\label{eq:K}
		K(x) = \frac{\pm 1}{4\pi\n{x}},
	\end{equation}
	which are the most relevant singular potentials from a physical viewpoint and the most challenging from a mathematical viewpoint.

	We consider the Hartree equation
	\begin{equation}\label{eq:Hartree}
		i\hbar\,\dpt \op = \com{H_{\op},\op} \quad \text{ with } \quad \op(0)= \op^{\init},
	\end{equation}
	where $\hbar=\frac{h}{2\pi}$ is the reduced Planck constant, $\op$ is a time-dependent positive self-adjoint trace class operator acting on $L^2(\Rd)$, and $H_{\op} = -\frac{\hbar^2\Delta}{2} + V_{\op}$ is the Hartree Hamiltonian. The operator $V_{\op}$ is the operator of multiplication by the mean-field potential $V_{\op}(x) = (K*\rho)(x)$ associated to the two-body interaction potential $K$ defined in our case by Equation~\eqref{eq:K}, and $\rho(x)$ is the quantum spatial density defined as the scaled restriction to the diagonal of the integral kernel of $\op$,
	\begin{equation*}
		\rho(x)=\diag(\op)(x) := h^3\op(x,x),
	\end{equation*}
	where we adopt the same notation to denote both the operator $\op$ and its integral kernel $\op(x,y)$. The classical analogue of Equation~\eqref{eq:Hartree} is the Vlasov--Poisson equation given by
	\begin{equation}\label{eq:Vlasov}
		\dpt f + \xi\cdot\Dx f + E_f \cdot\Dv f = 0 \quad \text{ with } \quad f(0, x, \xi)= f^{\init}(x, \xi)\ge 0,
	\end{equation}
	where $E_f = -\nabla V_f$ is the self-consistent force field associated to the mean-field potential $V_f(x)=(K*\rho_f)(x)$ with $\rho_f$ the spatial density given by
	\begin{equation*}
		\rho_f(x)=\intd f(t,x,\xi) \dd\xi.
	\end{equation*}
	
	For the well-posedness of the Hartree and the Vlasov--Poisson equations we refer to \cite{castella_solutions_1997} and \cite{pfaffelmoser_global_1992, lions_propagation_1991} respectively and references therein.
	
	Equation~\eqref{eq:Vlasov} can be seen as the semiclassical approximation of a system of many interacting quantum particles, as pointed out in the pioneering works by Narnhofer and Sewell~\cite{narnhofer_vlasov_1981} and by Spohn~\cite{spohn_vlasov_1981} where the Vlasov equation was obtained directly from the many-body Schr\"odinger equation with smooth interaction in the combined mean-field and semiclassical regime. This has been reconsidered in~\cite{graffi_meanfield_2003} and more recently in~\cite{chong_manybody_2021, chen_combined_2021}, where the case of the Coulomb potential with a $N$ dependent cut-off has been addressed. Moreover, a combined mean-field and semiclassical limit for particles interacting via the Coulomb potential has been treated in~\cite{golse_meanfield_2021} for factorized initial data whose first marginal is given by a monokinetic Wigner measure (that can be seen as the Klimontovich solutions to the Vlasov equation), which leads to the pressureless Euler--Poisson system.
	
	Most of the above mentioned works rely on compactness methods that do not allow for an explicit bound on the rate of convergence, which is essential for applications. For this reason the Hartree equation~\eqref{eq:Hartree} has been considered as an intermediate step to decouple the problem into two separate parts, namely to prove the convergence of the mean-field limit from the many-body Schr\"odinger equation towards the Hartree equation, and then the semiclassical limit from the Hartree equation to the Vlasov equation. In this paper, we are interested in the latter problem, that has been largely studied in different settings. It was first proven by Lions and Paul in~\cite{lions_mesures_1993}, and later in~\cite{markowich_classical_1993,figalli_semiclassical_2012}, that the Wigner transforms of the solutions of the Hartree equation~\eqref{eq:Hartree} converge in some weak sense to solutions of the Vlasov--Poisson equation. Quantitative rates of convergence were then obtained, first in the case when the Coulomb potential is replaced by a smoother potential, in Lebesgue-type norms~\cite{athanassoulis_strong_2011, amour_semiclassical_2013, benedikter_hartree_2016} and in a quantum analogue of the Wasserstein distances \cite{golse_schrodinger_2017}. The case of singular interactions was then treated in~\cite{lafleche_propagation_2019, lafleche_global_2021} with the same quantum Wasserstein distances, and in~\cite{saffirio_semiclassical_2019, saffirio_hartree_2020, lafleche_strong_2021} in Lebesgue-type norms. In particular, for $K$ as in Equation~\eqref{eq:K}, the explicit rate has been established in~\cite{lafleche_propagation_2019} for the weak topology and in~\cite{saffirio_hartree_2020, lafleche_strong_2021} for the Schatten norms.
	
	In a different setting, the semiclassical limit has also been studied for local perturbations of stationary states in the case of infinite gases in~\cite{lewin_hartree_2020}.

	Our goals here are twofold.~On the one hand, we want to improve the rate of convergence in the $L^2$ norm in terms of $\hbar$ for the convergence of the Wigner transform of the solution of the Hartree equation to the Vlasov--Poisson equation. The rate $\hbar^{3/4-\eps}$, for $\eps>0$ small enough, was obtained in~\cite{lafleche_strong_2021}, the reason being that the method used in that previous paper was relying on a $L^1$ weak-strong uniqueness principle, thus leading to the expected correct rate of order $\hbar$ in trace norm, but lower order rates for higher Schatten norms. In this paper we recover a rate of order $\hbar$ for the $L^2$ convergence, that is expected to be optimal.
	
	On the other hand, the method used in this paper is closer to the method used in~\cite{benedikter_meanfield_2016, chong_manybody_2021} to prove the semiclassical mean-field limit from the many-body Schr\"odinger equation towards the Hartree--Fock equation for mixed states, in the sense that it considers the square roots of the phase space densities in the $L^2$ setting. It might give additional insights on the problem of the derivation of the Vlasov--Poisson equation as a mean-field limit of a many-body problem. In fact, it sheds light on the difference of requirements between the $L^1$ and the $L^2$ method. In particular, the latter, still being a weak-strong estimate, seems however to require the uniform boundedness in the phase space of both the solutions of the Hartree and the Vlasov--Poisson equation, as it is clear from our main result.

	\subsection{Notations}

	\subsubsection{Functional spaces} Let us fix some notations before stating our main result. A point in the one-particle phase space $\Rd_x\times \Rd_{\xi}$ is denoted by $z = (x, \xi)$. We denote the phase space Lebesgue norm and its mixed norm variant by
	\begin{align*}
		\Nrm{f}{L^r} = \Big(\intdd \n{f(z)}^r\d z\Big)^{\!\frac{1}{r}}, \quad  \Nrm{f}{L^p_x L^q_{\xi}} = \bNrm{\!\Nrm{f(x, \cdot)}{L^q_{\xi}(\Rd)}\!}{L^p_x(\Rd)}
	\end{align*}
	for $1\le r<\infty$ and with the usual modification when $r=\infty$.  The corresponding weighted Sobolev spaces $W^{k,p}_n$ are defined by the norm 
	\begin{equation*}
		\Nrm{f}{W^{k,p}_{n}} = \Big(\sum_{\substack{ \n{\alpha} \le k}}\Nrm{\weight{\xi}^n\partial^\alpha f}{L^{p}}^2\Big)^{\!\frac{1}{2}}\quad \text{ with }\quad  
		\renewcommand\arraystretch{1.5}
		\begin{array}{l}
			\alpha \in \Z_{\ge 0}^6,\, \n{\alpha}= \sum^6_{\ii=1} \alpha_{\ii},
			\\
			\text{and } \bd^\alpha = \bd_{x_1}^{\alpha_1}\cdots \bd_{\xi_3}^{\alpha_6}
		\end{array}
	\end{equation*}
	where $\weight{z}:=\sqrt{1+|z|^2}$ is the standard bracket notation. We also use the standard notation $H^{k}_{n} = W^{k,2}_{n}$. 

	For any bounded linear operator $\op$ acting on $L^2(\Rd)$, we denote its operator norm by $\Nrm{\op}{\infty}$, and for $1\le p < \infty$, we define its Schatten-$p$ norm by 
	\begin{align}
		\Nrm{\op}{p}= \Tr{\n{\op}^p}^{\frac1p}
	\end{align}
	where $\n{\op} = \sqrt{\op^\ast \op}$. To obtain meaningful quantities in the semiclassical limit $\hbar\rightarrow 0$, we define the semiclassical analogue of the Schatten spaces via the rescaled Schatten norm
	\begin{align}
		\Nrm{\op}{\L^p} = h^{\frac3p} \Nrm{\op}{p}= h^{\frac3p}\Tr{\n{\op}^p}^{\frac1p}.
	\end{align}
	These semiclassical spaces play the roles of the Lebesgue spaces on the phase space. We will denote by $C$ a $\hbar$-independent constant that can change from line to line.

	\subsubsection{Weyl quantization and Wigner transform} The Fourier transform is defined with the convention that 
	\begin{align}
		\widehat g(\xi)=\intd e^{-2\pi i\, x\cdot \xi}\,g(x)\dd x.
	\end{align}
	Being consistent with our definition of the Fourier transform, we associate to every function $f \in L^1(\Rd_x\times\Rd_{\xi})$ the operator $\op_f$ with integral kernel 
	\begin{align}
		\op_f(x, y) = \intd e^{-2\pi i\, (y-x)\cdot \xi} f(\tfrac{x+y}{2}, h\xi)\dd \xi.
	\end{align}
	The operator $\op_f$ is called the Weyl quantization associated to $f(z)$. By the Fourier inversion theorem, we could also define the inverse mapping called the Wigner transform. More precisely, for any operator $\op$ with sufficiently regular kernel $\op(x, y)$, we define its Wigner transform by 
	\begin{align}
		f_{\op}(x, \xi) = \intd e^{-2\pi i\, \xi\cdot y/h} \op(x+\tfrac{y}{2}, x-\tfrac{y}{2})\dd y.
	\end{align}

	\subsection{Main result}
	Our main result reads as follows.
	\begin{theorem}\label{thm:main}
		Let $\op\geq 0$ be a solution to the Hartree equation~\eqref{eq:Hartree} with initial condition $\op^\init\in\L^1\cap\L^\infty$ and $f\geq 0$ be a solution to the Vlasov--Poisson equation~\eqref{eq:Vlasov} with initial condition $f^\init$ verifying
		\begin{equation*}
			f^\init,\,\sqrt{f^\init}\in W^{4,\infty}_4 \cap H^4_4 \quad \text{ and } \quad \iintd f^\init \n{\xi}^{n_1}\d x\dd\xi <\infty
		\end{equation*}
		for $n_1 > 6$. Then there exist $\Lambda, C_1, C_2 \in C^0(\R_+,\R_+)$, independent of $\hbar$ and  depending only on the initial conditions of equations~\eqref{eq:Hartree} and~\eqref{eq:Vlasov} such that
		\begin{equation}\label{eq:Wigner_main_estimate}
			\Nrm{f_{\op}-f}{L^2} \leq \cC_{\infty}^{\frac12} \(\bNrm{f_{\sqrt{\op^\init}} - \sqrt{f^\init}}{L^2} + C(t) \,\hbar\) e^{\Lambda(t)} + C_2(t)\,\hbar.
		\end{equation}
	\end{theorem}
	
	\begin{remark}
		Our method also allows us to treat the Hartree--Fock equation.
		In that case, the exchange term vanishes in the semiclassical limit and can be treated as an error term as in~\cite[Proposition~5.1]{lafleche_strong_2021}.
	\end{remark}
	
	\subsubsection{Strategy and explicit constants} Before giving more precise upper bounds on the functions $\Lambda$, $C_1$ and $C_2$, let us explain our strategy. The rationale of our result is a stability estimate for the square root of the solutions of the Vlasov equation in $L^2$, as explained in Section~\ref{sec:stability-est}.  {Rephrasing this estimate in the quantum context to estimate the difference between the solutions to the Hartree equation and the Vlasov equation}, we have to deal with the fact that the Weyl quantization of a non-negative function is not always non-negative. Thus, we have to consider an alternative quantization of $f$, sometimes called the Wick quantization, the anti-Wick quantization, the T\"oplitz quantization, the Husimi quantization or the coherent state quantization. We introduce the coherent state at the point $z$ and its corresponding density operator 
	\begin{align}
		\psi_z(y):= \(\pi \hbar\)^{-3/4} e^{-\n{y-x}^2/(2\hbar)}\,e^{i y\cdot \xi/\hbar} \quad \text{ and } \quad \op_z := h^{-3}\ket{\psi_z}\bra{\psi_z}.
	\end{align}
	Then, for every function $f$ on the phase space, we associate the operator $\widetilde{\op}_f$ defined by taking averages of the coherent states $\op_z$ against $f(z)$, that is
	\begin{align}\label{eq:wick}
		\widetilde{\op}_f = \intdd f(z)\,\op_z\dd z = \frac{1}{h^3}\intdd f(z) \ket{\psi_z}\bra{\psi_z}\d z. 
	\end{align}
	It follows from this formula that $\widetilde{\op}_f$ is a positive operator whenever $f\ge 0$. We summarize some of its other properties which we will need in this work in the following lemma. We refer the reader to \cite{lions_mesures_1993, lerner_fefferman-phong_2007} for additional properties of this quantization {and the proof of the following lemma}.

	\begin{lem}\label{lem:properties_of_Wick}
		Let $f$ be a function of the phase space. Then  the quantization~\eqref{eq:wick} is the combination of the Weyl quantization and a convolution with a Gaussian
		\begin{equation}\label{eq:Wick_convolution}
			\widetilde{\op}_f =\op_{\tilde f} \quad \text{ where } \quad \widetilde f = g_h\ast f \text{ with } g_h(z):= (\pi\hbar)^{-3}\,e^{-\n{z}^2/\hbar}.
		\end{equation}
		By \eqref{eq:Wick_convolution}, we have the estimate
		\begin{equation}\label{eq:Wick_Weyl_error}
			\Nrm{\op_f-\widetilde{\op}_f}{\L^2} = \Nrm{f-g_h\ast f}{L^2} \le \tfrac{3\, \hbar}{2}\Nrm{\nabla^2 f}{L^2}.
		\end{equation}
		For $1\le p\le \infty$, we have that 
		\begin{equation}\label{eq:Wick_Schatten}
			\Nrm{\widetilde{\op}_f}{\L^p}\le \Nrm{f}{L^p}.
		\end{equation}
	\end{lem}
	
	This quantization will allow us to consider an intermediate equation between the Hartree and the Vlasov--Poisson equations, depending on the mean-field force of the Vlasov equation, but whose solution is a positive operator. More precisely, we consider $\widetilde{\op}$ a solution to the equation
	\begin{equation*}
		i\hbar\,\dpt\widetilde{\op} = \com{H_f,\widetilde{\op}}
	\end{equation*}
	with Hamiltonian $H_f=-\frac{\hbar^2}{2}\,\Delta+V_{f}$ and initial data $\widetilde{\op}^\init := \widetilde{\op}_{\sqrt{f^\init}}^2$, and we define 
	\begin{equation*}
		\tilde v := \sqrt{\widetilde{\op}},
	\end{equation*}
	which satisfies the same equation. Our estimates rely on the semiclassical regularity of this equation.
	
	\subsubsection{Quantum Sobolev spaces} To define the semiclassical version of Sobolev spaces on the phase space, we introduce the following operators
	\begin{equation}
		\Dhx \op := \com{\nabla,\op} \quad \text{ and } \quad
		\Dhv \op := \com{\frac{x}{i\hbar},\op},
	\end{equation}
	which can be seen as an application of the correspondence principle of quantum mechanics. More precisely, one observes that these operators correspond to the gradients of the Wigner transform, that is
	\begin{equation}
		f_{\Dhx \op} = \Dx f_{\op}\quad \text{ and } \quad f_{\Dhv \op} = \Dv f_{\op}.
	\end{equation}
	We will refer to $\Dhx\op$ and $\Dhv\op$ as the quantum gradients. The semiclassical analogues of the weighted kinetic homogeneous Sobolev norms are defined by
	\begin{subequations}
		\begin{align}
			\Nrm{\op}{\dot{\cW}^{k,p}(\opm)}^p &:= \sum_{\n{\alpha}\le k} \Nrm{\Dh_{\!z}^\alpha\op}{\L^p(\opm)}^p\quad \text{ with } \quad \Dh_{\!z}^\alpha = \Dh_{\!x_1}^{\alpha_1}\cdots\Dh_{\!\xi_3}^{\alpha_6}
			\\
			\Nrm{\op}{\dot{\cW}^{k,\infty}(\opm)} &:= \sup_{\n{\alpha} \le k}  \Nrm{\Dh_{\!z}^\alpha\op}{\L^\infty(\opm)},
		\end{align}
	\end{subequations}
	and we consider the particular case of the weights defined for $n\in \N$ by
	\begin{align}\label{eq:weight}
		\opm := \weight{\opp}^n = \big(1 + \n{\opp}^2\big)^{\!\frac{n}{2}},
	\end{align}
	where $\opp = -i\hbar\nabla$ so that $\n{\opp}^2 = -\hbar^2\Delta$. We also define the inhomogeneous version by
	\begin{align}
		\Nrm{\op}{\cW^{k,p}(\opm)}^p &:= \Nrm{\op}{\L^p(\opm)}^p + \Nrm{\op}{\dot{\cW}^{k,p}(\opm)}^p,
	\end{align}
	with the usual modification when $p=\infty$. In particular, for $p=2$, $\Nrm{\op}{\cW^{k,2}} = \Nrm{f_{\op}}{H^k}$.
	
	\subsubsection{A precised version of the main theorem} With the above definitions, we now give another version of the main theorem with more explicit upper bounds on the constants appearing in Inequality~\eqref{eq:Wigner_main_estimate}.
	
	\begin{theorem}\label{thm:main_2}
		Under the assumptions of Theorem~\ref{thm:main}, the following estimate holds
		\begin{equation}\label{eq:Weyl_main_estimate}
			\Nrm{\op-\op_f}{\L^2} \leq \cC_{\infty}^{\frac12} \(\Nrm{\widetilde{\op}_{\sqrt{f^\init}} - \sqrt{\op^\init}}{\L^2} + C(t) \,\hbar\) e^{\Lambda(t)} + C_2(t)\,\hbar.
		\end{equation}
		An upper bound for $\Lambda$ is given by $\Lambda(t) \leq C \int_0^t \lambda(s)\dd s$ where $\lambda$ is defined for some $\eps\in(0,1)$ by
		\begin{equation}\label{eq:lambda}
			\lambda(t) = \Nrm{\Dhv \tilde{v}}{\cW^{1,2}} \Nrm{\rho_f}{L^\infty_x}^{\frac12} + \cC_{\infty}^{\frac12} \Nrm{\Dhv \tilde{v}}{\L^{3\pm\eps}(\weight{\opp}^n)}.
		\end{equation}
		with $n>2$. The functions $C(t)$ and $C_2(t)$ are bounded above by 
		\begin{align}\label{eq:bound-c}
			C(t)^2 &\leq C \int_0^t \Nrm{\Dhv \tilde{v}}{\cW^{1,2}}^2 \(C^\init + \Nrm{\rho_f}{W^{1,\infty}_x} \Nrm{\op_f}{\cW^{2,2}(\weight{\opp}^2)}\)^2
			\\
			C_2(t) &\leq C \(C^\init + \int_0^t \Nrm{\rho_f}{W^{1,\infty}_x} \Nrm{\nabla_{\xi}^2 f}{L^2}\)\label{eq:bound-c2}
		\end{align}
		which remain bounded at any time $t\geq 0$, and where
		\begin{equation}\label{eq:C_init}
			C^\init = \bNrm{\sqrt{f^\init}}{W^{1,4}_1\cap H^3}^2 + \Nrm{f^\init}{H^1_1\cap H^2}.
		\end{equation}
		\end{theorem}
		
		The fact that these quantities~\eqref{eq:lambda}, \eqref{eq:bound-c} and \eqref{eq:bound-c2} remain bounded uniformly in $\hbar$ at any time is proved in the last section in Proposition~\ref{prop:propagation_regularity}.

	\begin{remark}
		Notice that the quantity
		\begin{equation*}
			\Nrm{\widetilde{\op}_{\sqrt{f^\init}} - \sqrt{\op^\init}}{\L^2}= \Nrm{\op_{g_h*\sqrt{f^\init}} - \sqrt{\op^\init}}{\L^2} = \Nrm{g_h*\sqrt{f^\init} - f_{\sqrt{\op^\init}}}{L^2}.
		\end{equation*}
		that appears in the right-hand side of Equation~\eqref{eq:Weyl_main_estimate} is $0$ in the case when $\op^\init = \widetilde{\op}_{\sqrt{f^\init}}^2$, that is when $\op^\init$ is the square of a Wick quantization.\\
		Due to the strategy of the proof, that mimics the proof of the stability estimates presented in Section~\ref{sec:stability-est}, the initial datum for the auxiliary problem \eqref{eq:linear-Hartree} has to be chosen positive, close  to $\op_f^\init$ in $\L^2$, its square root is close to $\sqrt{\op^\init}$ in $\cL^2$ and regular in the sense that the quantity
		 $\bNrm{\Dhv{\sqrt{\smash[t]{\widetilde{\op}}^{\init}}}\,\opm}{\L^p} + \bNrm{\Dhx{\sqrt{\smash[t]{\widetilde{\op}}^{\init}}}\,\opm}{\L^p}$ is uniformly bounded with respect to~$\hbar$. The choice $\widetilde{\op}^\init = \widetilde{\op}_{\sqrt{f^\init}}^2$ guarantees these properties.
	\end{remark}

\section{Stability estimates for the Vlasov--Poisson and Hartree equations}\label{sec:stability-est}

\subsection{Classical case}
	
	We start by explaining our method through examining the classical case and obtaining a stability estimate for the Vlasov--Poisson equation.

	\begin{theorem}
		Let $f_1$ and $f_2$ be two solutions of the Vlasov--Poisson equation verifying $\Nrm{f_1}{L^\infty} = \Nrm{f_2}{L^\infty} \leq \cC_{\infty}$. Then we have the following bound
		\begin{equation*}
			\BNrm{\sqrt{f_1} - \sqrt{f_2}}{L^2} \leq \BNrm{\sqrt{f_1^\init} - \sqrt{f_2^\init}}{L^2} e^{\Lambda(t)}
		\end{equation*}
		where $\Lambda(t) = C\,\int_0^t \lambda(s)\dd s$ with
		\begin{equation*}
			\lambda(t) = \Nrm{\rho_2}{L^{\infty}_x}^{\frac12} \Nrm{\Dv \sqrt{f_2}}{L^3_xL^2_\xi} + \cC_{\infty}^{\frac12} \Nrm{\Dv \sqrt{f_2}}{L^{3,1}_xL^1_\xi}.
		\end{equation*}
		In particular, we have the following estimate
		\begin{equation*}
			\Nrm{f_1-f_2}{L^2} \leq 2\, \cC_{\infty}^{\frac12} \Nrm{f_1-f_2}{L^1}^{\frac12} e^{\Lambda(t)}.
		\end{equation*}
	\end{theorem}

	\begin{proof}
		Let $v_1=\sqrt{f_1}$, $v_2=\sqrt{f_2}$ and $v := v_1-v_2$. Then $v$ satisfies the equation
		\begin{equation*}
			\big(\dpt+\xi\cdot\Dx+E_{v_1^2}\cdot\Dv\big)v = \big(E_{v_2^2}-E_{v_1^2}\big)\cdot\Dv v_2 = \(\nabla K * \rho_{\(v_2+v_1\)v}\)\cdot\Dv v_2.
		\end{equation*}
		Then, by direct computation, we have that
		\begin{align*}
			\frac{1}{2}\frac{\d}{\d t}\Nrm{v}{L^2}^2 &= -\intdd v \(\nabla K * \rho_{\(v_2+v_1\)v}\)\cdot\Dv v_2\dd x\dd\xi
			\\
			&= -\int_{\R^{12}} \big(\n{v(z')}^2 + 2\, v_2(z')\,v(z')\big) \nabla K(x-x') \cdot \(v(z)\, \Dv v_2(z)\) \dd z\dd z'\\
            &=: I_1 + 2 \, I_2.
		\end{align*}
		The first term is bounded by writing first
		\begin{align*}
			I_1 &\leq \BNrm{\nabla K * \intd v\, \Dv v_2 \dd\xi}{L^\infty_x} \Nrm{v}{L^2}^2 \leq \BNrm{\n{\nabla K} * \intd \n{\Dv v_2} \d\xi}{L^\infty_x} \Nrm{v}{L^\infty} \Nrm{v}{L^2}^2.
		\end{align*}
		and then by applying H\"older's inequality for the Lorentz spaces to get 
		\begin{equation*}
			I_1 \leq \Nrm{\nabla K}{L^{3/2,\infty}_x} \Nrm{\Dv v_2}{L^{3,1}_x L^1_\xi} \Nrm{v}{L^\infty} \Nrm{v}{L^2}^2.
		\end{equation*}
		The second term is bounded by the Hardy--Littlewood--Sobolev inequality and then by the Cauchy--Schwarz inequality, leading to
		\begin{align*}
			I_2 &\leq C\, \BNrm{\intd v\,v_2\dd \xi\,}{L^2_x}\,\BNrm{\intd v\,\Dv v_2 \dd \xi\,}{L^{6/5}_x}
			\\
			&\leq C \Nrm{\Nrm{v}{L^2_\xi}\Nrm{v_2}{L^2_\xi}}{L^{2}_x} \Nrm{\Nrm{v}{L^2_\xi}\Nrm{\Dv v_2}{L^2_\xi}}{L^{6/5}_x}.
		\end{align*}
		Finally, by H\"older's inequality, we obtain
		\begin{align*}
			I_2 &\leq C \Nrm{v_2}{L^\infty_xL^2_\xi}\Nrm{\Dv v_2}{L^3_xL^2_\xi} \Nrm{v}{L^2}^2 = C \Nrm{\rho_2}{L^\infty_x}^{\frac12} \Nrm{\Dv v_2}{L^3_xL^2_\xi} \Nrm{v}{L^2}^2.
		\end{align*}
		Combining the bounds for $I_1$ and $I_2$ leads to
		\begin{equation*}
			\frac{\d}{\d t} \Nrm{v}{L^2}^2 \leq C\, \big(\Nrm{\rho_2}{L^\infty_x} \Nrm{\Dv v_2}{L^3_xL^2_\xi} + \Nrm{v}{L^\infty} \Nrm{\Dv v_2}{L^{3,1}_x L^1_\xi}\big) \Nrm{v}{L^2}^2.
		\end{equation*}
		We conclude by using the fact that $\Nrm{v}{L^\infty} \leq \Nrm{f_1}{L^\infty}^{\frac12} + \Nrm{f_2}{L^\infty}^{\frac12} = 2\, \cC_{\infty}^{\frac12}$ and Gr\"onwall's lemma.
	\end{proof}
	
\subsection{Quantum case}\label{sect:Qcase}

	Using the quantum analogue of gradients and Lebesgue norms of the phase space, we deduce a similar estimate for the Hartree equation.

	\begin{theorem}\label{thm:Q-uniqueness}
		Let $\op_1$ and $\op_2$ be two solutions of the Hartree equation~\eqref{eq:Hartree} such that $\Nrm{\op_1}{\L^\infty} \leq \cC_{\infty}$ and $\Nrm{\op_2}{\L^\infty} \leq \cC_{\infty}$ and let $v_1 = \sqrt{\op_1}$ and $v_2 = \sqrt{\op_2}$. Then there exists a universal constant $C>0$ and $T>0$ such that for any $t\in [0,T]$,
		\begin{equation}\label{eq:L^2-L^2}
			\nrm{v_1-v_2}_{\L^2} \leq \Nrm{v_1^\init-v_2^\init}{\L^2} e^{\lambda(t)}
		\end{equation}
		where $\Lambda(t) = C \int_0^t \lambda(s)\dd s$ with $\lambda$ given for some $\eps\in(0,1)$ by
		\begin{equation}\label{eq:lambda_2}
			\lambda(t) = \Nrm{\Dhv v_2}{\cW^{1,2}} \Nrm{\rho_f}{L^\infty_x}^{\frac12} + \cC_{\infty}^{\frac12} \Nrm{\Dhv v_2\,\opm_n}{\L^{3\pm\eps}},
		\end{equation}
		where $\opm_n = \weight{\opp}^n$ with $n>2$. This implies the following estimate
		\begin{equation}\label{eq:L^2-L^1}
			\Nrm{\op_1-\op_2}{\L^2} \leq 2\, \cC_{\infty}^{\frac12} \Nrm{\op_1^\init-\op_2^\init}{\L^1}^{\frac12} e^{\Lambda(t)}.
		\end{equation}
	\end{theorem}
	
	\begin{remark}
		It is actually not difficult to see from the interpolation argument in the proof of Lemma~\ref{lem:commutator_estimate_HS} that $\Nrm{\Dhv v_2}{\cW^{1,2}} = \Nrm{\Dv f_{v_2}}{H^1}$ can actually be replaced by $\Nrm{\Dv f_{v_2}}{H^{1/2}}$.
	\end{remark}
	
	\begin{remark}
		Contrarily to our main theorem, the result here is only local-in-time since it is not yet known whether the global-in-time uniform in $\hbar$ propagation of regularity holds for the Hartree equation in the case of the Coulomb potential. For slightly less singular potentials, we have proved it in our recent paper~\cite{chong_globalintime_2022}.
	\end{remark}

	\begin{proof}
		Let $v = v_1-v_2$. Then $v$ satisfies the equation
		\begin{equation*}
			i\hbar\,\dpt v = \com{H_{\op_1}, v} + \com{V_{\op_1}-V_{\op_2}, v_2}.
		\end{equation*}
		Differentiating its Hilbert--Schmidt norm with respect to time, we obtain
		\begin{align*}
			\frac{1}{4\pi}\, \frac{\d}{\d t} \Nrm{v}{\L^2}^2 &= h^2 \Tr{\com{V_{\op_1}-V_{\op_2}, v_2} v}
			= h^2 \Tr{\com{V_{v_1^2-v_2^2}, v_2} v}.
		\end{align*}
		Since $v_1^2 - v_2^2 = \(v+v_2\)^2 - v_2^2 = v^2 + v\,v_2 + v_2\,v$ and $\Diag{v\,v_2} = \Diag{v_2\, v}$, we get
		\begin{align*}
		 	\frac{1}{4\pi}\,\frac{\d}{\d t}  \Nrm{v}{\L^2}^2 = h^2  \Tr{\com{V_{v^2}, v_2} v} + 2\,h^2 \Tr{\com{V_{v_2\,v}, v_2} v} =: I_1 + 2\, I_2.
		\end{align*}
		To bound the first term, we introduce the notation $K_x := K(x-y)$ and write
		\begin{align*}
			I_1 = h^2  \intd \Tr{\com{K_x,v_2} v} \rho_{v^2}(x) \dd x &\leq \sup_{x\in\Rd} \(h^{-1}\Nrm{\com{K_x,v_2}}{\L^1} \Nrm{v}{\L^\infty}\) \intd \n{\rho_{v^2}} \d x
			\\
			&\leq C \Nrm{\Dhv{v_2}\,\opm_n}{\L^{3\pm\eps}} \Nrm{v}{\L^\infty} \Nrm{v}{\L^2}^2
		\end{align*}
		where we have used \cite[Theorem 4.1]{lafleche_strong_2021}  to bound the commutator in trace norm. To bound the second term, we use H\"older's inequality to get
		\begin{equation*}
			I_2 \leq \Nrm{\com{V_{v\,v_2}, v_2}}{\L^2} \Nrm{v}{2}
		\end{equation*}
		and then we apply the following inequality which we will prove in the subsequent section (see Lemma~\ref{lem:commutator_estimate_HS}) 
		\begin{equation*}
			\Nrm{\com{V_{v\,v_2},v_2}}{\L^2} \leq C \Nrm{\Dhv v_2}{\cW^{1,2}} \Nrm{\rho_{v\,v_2}}{L^2_x}.
		\end{equation*}
		Also, notice that
		\begin{align*}
			\Nrm{\rho_{v\,v_2}}{L^2_x} &= h^3 \(\intd \Big|\intd v(x,y)\,v_2(y,x)\dd y\Big|^2\d x\)^\frac{1}{2} \leq h^\frac{3}{2} \Nrm{v_2}{L^\infty_xL^2_y} \Nrm{v}{\L^2}.
		\end{align*}
		Hence, since $h^3\Nrm{v_2}{L^2_y}^2 = \Diag{\op_2} = \rho_2$, we deduce that
		\begin{equation*}
			I_2 \leq \Nrm{\Dhv v_2}{\cW^{1,2}} \Nrm{\rho_2}{L^\infty_x}^{\frac12} \Nrm{v}{\L^2}^2,
		\end{equation*}
		and, as in previous proposition, applying Gr\"onwall's lemma proves~\eqref{eq:L^2-L^2}. Inequality~\eqref{eq:L^2-L^1} follows from the identity $\op_1-\op_2=\frac{1}{2}(v_1-v_2)(v_1+v_2)+\frac{1}{2}(v_1+v_2)(v_1-v_2)$ and the Powers--St{\o}rmer inequality (see~\cite{powers_free_1970}).
	\end{proof}
	
	\subsubsection{A Semiclassical inequality for commutators}

	\begin{lem}\label{lem:commutator_estimate_HS}
		There exists a constant $C>0$ such that for any self-adjoint trace class operators $\op$ and $\opmu$, we have the following estimate
		\begin{equation*}
			\frac{1}{\hbar} \Nrm{\com{V_{\op},\opmu}}{\L^2} \leq C \Nrm{\rho}{L^2_x} \Nrm{\Dhv \opmu}{\cW^{1,2}}
		\end{equation*}
		where $\rho = \Diag{\op}$.
	\end{lem}
	
	\begin{proof}[Proof of Lemma~\ref{lem:commutator_estimate_HS}]
		We use a first order Taylor expansion of $V_{\op}$ of the form
		\begin{equation*}
			V_{\op}(x) = V_{\op}(y) + \(x-y\)\cdot \int_0^1 \nabla V_{\op}(z_\theta)\dd\theta
		\end{equation*}
		where $z_\theta = \(1-\theta\)x+\theta\, y$. This implies that $\frac{1}{i\hbar}\com{V_{\op},\opmu}$ is the operator with kernel
		\begin{equation*}
			\frac{1}{i\hbar}\com{V_{\op},\opmu}(x,y) = \int_0^1 \nabla V_{\op}(z_\theta) \cdot \Dhv \opmu(x,y) \dd\theta.
		\end{equation*}
		Its Hilbert--Schmidt norm being given by the $L^2$ norm of its kernel, and doing the change of variables $(x,y) \mapsto (x+\theta y, x-(1-\theta) y)$ with Jacobian equal to $1$, we obtain
		\begin{align*}
			\Nrm{\tfrac{1}{i\hbar}\com{V_{\op},\opmu}}{\L^2}^2 &\leq h^d\int_0^1\!\!\! \iintd \n{\nabla V_{\op}(x) \cdot \Dhv \opmu(x+\theta y,x-(1-\theta) y)}^2\d x\dd y \dd\theta
			\\
			&\leq h^d \Nrm{\nabla V_{\op}}{L_x^6}^2 \int_0^1\!\!\! \intd \Nrm{\Dhv \opmu(x+\theta y,x-(1-\theta) y)}{L^3_x}^2 \d y \dd\theta,
		\end{align*}
		where the second inequality follows by H\"older's inequality. The first factor on the right-hand side is controlled using the Hardy--Littlewood--Sobolev inequality by $\Nrm{\rho}{L^2_x}$. To control the second factor, we perform the change of variable $x\mapsto x-\theta y$ to get
		\begin{equation*}
			\int_0^1\!\!\! \intd \Nrm{\Dhv \opmu(x+\theta y,x-(1-\theta) y)}{L^3_x}^2 \d y \dd\theta = \intd \Nrm{\Dhv \opmu(x,x-y)}{L^3_x}^2 \d y.
		\end{equation*}
		The $L^3$ norm in this term is now bounded using the Gagliardo--Nirenberg--Sobolev inequality associated to the embedding $H^1\subset L^3$ which yields
		\begin{equation*}
			\Nrm{\Dhv \opmu(x,x-y)}{L^3_x}^2 \leq C \(\Nrm{\Dhv\opmu(x,x-y)}{L^2_x}^2 + \Nrm{\Dx(\Dhv \opmu(x,x - y))}{L^2_x}^2\).
		\end{equation*}
		However, notice now that the gradient with respect to $x$ appearing in the above formula is actually also a quantum gradient of the operator since
		\begin{equation*}
			\nabla_x(\Dhv\opmu(x,x-y)) = ((\nabla_1+\nabla_2)\Dhv\opmu)(x,x-y) = \com{\nabla,\Dhv{\opmu}}(x,x-y)
		\end{equation*}
		and this is exactly $\Dhx{\Dhv{\opmu}}(x,x-y)$. Using the fact that for both $\opnu = \Dhv\opmu$ and for $\opnu = \Dhx\Dhv\opmu$, by the Fubini theorem and the change of variables $y\mapsto x-y$ it holds
		\begin{align*}
			h^d \intd \Nrm{\opnu(x,x-y)}{L^2_x}^2 \d y = h^d \iintd \n{\opnu(x,x-y)}^2 \d x\dd y = \Nrm{\opnu}{\L^2}^2,
		\end{align*}
		we therefore arrive at the inequality
		\begin{equation*}
			\int_0^1\!\!\! \intd \Nrm{\Dhv \opmu(x+\theta y,x-(1-\theta) y)}{L^3_x}^2 \d y \dd\theta \leq C \Nrm{\Dhv \opmu}{\cW^{1,2}},
		\end{equation*}
		which concludes the proof.
	\end{proof}
	\medskip

\section{Semiclassical limit in \texorpdfstring{$\L^2$}{L2} Schatten norm}

	We consider $\op$ a solution of the Hartree equation~\eqref{eq:Hartree} and $f$ a solution of the Vlasov--Poisson equation~\eqref{eq:Vlasov}. Then, in the same spirit as~\cite{lafleche_strong_2021}, we will use the fact that the Weyl transform $\op_f$ of $f$ solves the equation
	\begin{equation}\label{eq:Weyl-Vlasov}
		i\hbar\,\dpt \op_f = \com{H_{f},\op_f} - B_f(\op_f) \quad \text{ with } \quad \op_f(0) = \op^\init_{f} := \op_{f^\init}
	\end{equation}
	where $B_f(\op_f)$ is the operator with kernel
	\begin{equation}\label{def:B_operator}
		B_f(\op_f)(x,y) = \(V_{f}(x) - V_{f}(y) - \(x-y\)\cdot\nabla V_{f}\Big(\frac{x+y}{2}\Big)\!\) \op_f(x,y).
	\end{equation}
	By the proof of \cite[Proposition~4.4]{lafleche_strong_2021}, we see that $B_f(\op_f)$ verifies the estimate
	\begin{equation}\label{eq:B_bound}
		\frac{1}{\hbar}\Nrm{B_f(\op_f)}{\L^2} \leq C\,\hbar\Nrm{\nabla E_f}{L^\infty_x} \Nrm{\nabla_\xi^2 f}{L^2},
	\end{equation}
	where $\Nrm{\nabla E_f}{L^\infty_x}$ is controlled by $\Nrm{\rho_f}{W^{1,\infty}}$. Let $\op$ be a solution to the Cauchy problem \eqref{eq:Hartree} and $\widetilde{\op}$ be a solution to the following Cauchy problem for the linear Hartree equation
	\begin{equation}\label{eq:linear-Hartree}
		i\hbar\,\dpt\widetilde{\op} = \com{H_f,\widetilde{\op}} \quad \text{ with } \quad \widetilde{\op}(0) = \widetilde{\op}^\init := \widetilde{\op}_{\sqrt{f^\init}}^2
	\end{equation}
	with Hamiltonian $H_f=-\frac{\hbar^2}{2}\,\Delta+V_{f}$. 
	
	\begin{remark}
		The choice $\widetilde{\op}^\init = \widetilde{\op}_{\sqrt{f^\init}}^2$ guarantees the following property. Defining $\tilde{v}^\init = \widetilde{\op}_{\sqrt{f^\init}}$, then it holds
		\begin{equation*}
			\Nrm{\Dh_{\!\eta} \tilde{v}^\init \,\opm}{\L^p} = \bNrm{\op_{g_h*\nabla_\eta\sqrt{f^\init}}\,\opm}{\L^p}
		\end{equation*}
		which, by \cite[Proposition 3.2]{lafleche_strong_2021}, is bounded by the regularity of $\sqrt{f^{\init}}$.
	\end{remark}
	
	Our aim is to show that the bound on the difference between $\op$ and $\op_f$ in $\L^2$ is of order $\hbar$. To this end, we consider $\widetilde{\op}$ the solution to the auxiliary problem \eqref{eq:linear-Hartree} with initial datum $\widetilde{\op}^\init = \widetilde{\op}_{\sqrt{f^\init}}^2$. Then, applying Minkowski's inequality yields
	\begin{equation}\label{eq:op-opf}
		\Nrm{\op-\op_f}{\L^2} \leq \Nrm{\op-\widetilde{\op}}{\L^2} + \Nrm{\widetilde{\op}-\op_f}{\L^2}.
	\end{equation}
	The second term on the right-hand side corresponds to a measure of the defect of positivity of the Weyl quantization of $f$. From the equations verified by $\widetilde{\op}$ and $\op_f$ and Inequality~\eqref{eq:B_bound}, it is bounded by Gr\"onwall's lemma by
	\begin{equation}\label{eq:optilde-opf}
		\Nrm{\widetilde{\op}-\op_f}{\L^2} \leq \Nrm{\widetilde{\op}^\init-\op_{f^\init}}{\L^2} + C\, \hbar \int_0^t\Nrm{\nabla E_f(s,\cdot)}{L^\infty_x}\Nrm{\nabla_{\xi}^2 f(s,\cdot)}{L^2} \d s,
	\end{equation}
	The first term on the right-hand side of Inequality~\eqref{eq:op-opf} is the true nonlinear error that corresponds to the stability estimate for the Hartree equation proved in the previous section. It can be readily estimated in a similar manner as in Section~\ref{sect:Qcase} by
	\begin{equation}\label{eq:sqrt-op_optilde}
		\Nrm{\op-\widetilde{\op}}{\L^2} \leq \bNrm{\sqrt{\op}-\sqrt{\widetilde{\op}}\,}{\L^2} \(\Nrm{\sqrt{\op}}{\L^\infty} + \bNrm{\sqrt{\widetilde{\op}}\,}{\L^\infty}\).
	\end{equation}
	We now estimate each term separately.

\subsection{An estimate on the lack of positivity}

	For the second term on the right-hand side of Equation \eqref{eq:optilde-opf}, $\Nrm{\nabla E_f}{L^\infty_x}$ and $\Nrm{\nabla^2_\xi f}{L^2}$ are bounded by \cite[Proposition~A.1]{lafleche_strong_2021}. As for the first term, the difference between the initial data $\widetilde{\op}^\init$ and $\op_f^\init$ can be written as follows
	\begin{equation*}
		\widetilde{\op}^\init - \op_f^\init = \widetilde{\op}_{\sqrt{f^\init}}^2 - \widetilde{\op}_f^\init + \widetilde{\op}_f^\init - \op_f^\init.
	\end{equation*}
	Since, by Inequality~\eqref{eq:Wick_Weyl_error}, we have that 
	\begin{equation*}
		\bNrm{\widetilde{\op}_f^\init - \op_f^\init}{\L^2} \leq \frac{3\,\hbar}{2} \Nrm{\nabla f^\init}{L^2},
	\end{equation*}
	it remains to estimate $\widetilde{\op}_{\sqrt{f^\init}}^2 - \widetilde{\op}_f^\init$ in $\L^2$. This is done via the following lemma.
	
	\begin{lem}\label{lem:commut_square_wick}
		Let $g\in  W^{1,2p}$ with $p\in[1,\infty]$. Then
		\begin{align*}
			\Nrm{\widetilde{\op}_{g^2} - \widetilde{\op}_g^2}{\L^p} \leq 48\,\hbar \Nrm{\nabla g}{L^{2p}}^2.
		\end{align*}
	\end{lem}
	
	\begin{proof}
		Notice that for any $(f,g)\in (L^1+ L^\infty)^2$,
		\begin{equation*}
			\widetilde{\op}_{f}\,\widetilde{\op}_{g} = h^{-6} \iintd f\,g' \,G(z,z') \, \ket{\psi_{z}}\bra{\psi_{z'}} \dd z \dd z'
		\end{equation*}
		where $f = f(z), g' = g(z')$ and 
		\begin{equation*}
			G(z,z') = \Inprod{\psi_z}{\psi_{z'}} = \(\tfrac{2}{h}\)^{\frac32} \intd e^{-\(\n{y-x}^2+\n{y-x'}^2\)/(2\hbar)}\, e^{i\,y\cdot\(\xi'-\xi\)/\hbar} \dd y.
		\end{equation*}
		Using the parallelogram identity, we can evaluate the integral in $y$ of $G$ as follows 
		\begin{align*}
			G(z,z') &= \(\tfrac{2}{h}\)^{\frac32} e^{-\n{\frac{x-x'}{2}}^2/\hbar} \intd e^{-\n{y-\frac{x+x'}{2}}^2/\hbar}\, e^{i\,y\cdot\(\xi'-\xi\)/\hbar} \dd y\\
			&= \(\tfrac{2}{h}\)^{\frac32} e^{-\n{\frac{x-x'}{2}}^2/\hbar} e^{i\,\(x+x'\)\cdot\(\xi'-\xi\)/(2\hbar)} \F{e^{-2\pi\n{y}^2/h}}(\tfrac{\xi-\xi'}{h})
			\\
			&= e^{-\n{\frac{z-z'}{2}}^2/\hbar} e^{i\,\(x+x'\)\cdot\(\xi'-\xi\)/(2\hbar)}.
		\end{align*}
		Notice also that $\widetilde{\op}_1 = \mathds{1}$ is the identity operator. Hence, we write
		\begin{align*}
			\widetilde{\op}_{fg} &= \frac{\widetilde{\op}_{fg}\, \widetilde{\op}_1 + \widetilde{\op}_1\, \widetilde{\op}_{fg}}{2} = h^{-6} \iintd \frac{f'g'+fg}{2} \,G(z,z') \ket{\psi_{z}}\bra{\psi_{z'}} \d z \dd z'.
		\end{align*}
		Therefore, since $f\,g + f'\,g' - f\,g' - f'\,g = \(f-f'\) \(g-g'\)$, we obtain
		\begin{align*}
			\Lambda(f,g) &:= \widetilde{\op}_{fg} - \frac{\widetilde{\op}_f\,\widetilde{\op}_g + \widetilde{\op}_g\,\widetilde{\op}_f}{2}
			\\
			&= h^{-6} \iintd \frac{\(f-f'\)\(g-g'\)}{2} \,G(z,z')  \ket{\psi_{z}}\bra{\psi_{z'}} \d z \dd z'.
		\end{align*}
		To obtain bounds on $\Lambda(g,g) = \widetilde{\op}_{g^2} - \widetilde{\op}_g^2$, we will now bound $\Lambda(f,g)$ by bilinear interpolation.
		
		\step{Trace norm estimate} Using Minkowski's inequality and the fact that the trace norm of $\ket{\psi_{z}}\bra{\psi_{z'}}$ is 1, we obtain the following trace norm estimate 
		\begin{multline*}
			\Nrm{\Lambda(f,g)}{\L^1} \leq \frac{h^{-3}}{2} \iint_{[0,1]^2}\iintdd \n{\nabla f(z_\theta)} \n{\nabla g(z_{\theta'})} G_2(z-z') \dd z \dd z' \dd \theta\dd \theta'
			\\
			\leq \frac{h^{-3}}{2} \(\iintdd \n{\nabla f(u)}^2 G_2(v) \dd u \dd v\)^\frac{1}{2} \(\iintdd \n{\nabla g(u)}^2 G_2(v) \dd u \dd v\)^\frac{1}{2}
			\\
			\leq 48\,\hbar \Nrm{\nabla f}{L^2} \Nrm{\nabla g}{L^2}
		\end{multline*}
		where $G_2(z) = \n{z}^2 e^{-\n{\frac{z}{2}}^2/\hbar}$ and $z_\theta = \(1-\theta\) z + \theta\, z'$.
		\step{Operator norm estimate} Let $(\varphi,\phi)\in (L^2)^2$ and define $\psi_\varphi(z) := \Inprod{\psi_z}{\varphi}$. Then
		\begin{align*}
			\Inprod{\phi}{\Lambda(f,g)\,\varphi} &= h^{-6} \iintdd \frac{\(f-f'\)\(g-g'\)}{2} \,G(z,z') \, \psi_{\phi}(z)\,\psi_{\varphi}(z') \dd z \dd z'
			\\
			&\leq h^{-6} \Nrm{\nabla f}{L^{\infty}}\Nrm{\nabla g}{L^{\infty}}\iintdd G_2(z-z') \n{\psi_{\phi}(z)}\n{\psi_{\varphi}(z')} \d z \dd z'
			\\
			&\leq h^{-6} \Nrm{\nabla f}{L^{\infty}}\Nrm{\nabla g}{L^{\infty}}\Nrm{G_2}{L^1} \Nrm{\psi_{\phi}}{L^2}\Nrm{\psi_{\varphi}}{L^2}
		\end{align*}
		where the last inequality follows from Young's inequality. Observing that $\psi_\varphi(z) = \F{\psi_0(x-\cdot)\,\varphi}(\xi/h)$, we deduce that $\Nrm{\psi_\varphi}{L^2} = h^{\frac32} \Nrm{\varphi}{L^2}$ and so 
		\begin{equation*}
			\Nrm{\Lambda(f,g)}{\L^\infty} \leq 48\,\hbar \Nrm{\nabla f}{L^{\infty}}\Nrm{\nabla g}{L^{\infty}}.
		\end{equation*}
		
		\step{General Schatten norms} By the complex bilinear interpolation (see e.g. \cite[Section~4.4]{bergh_interpolation_1976}), we deduce that for any $p\in [1,\infty]$,
		\begin{equation*}
			\Nrm{\Lambda(f,g)}{\L^p} \leq 48\,\hbar \Nrm{\nabla f}{L^{2p}}\Nrm{\nabla g}{L^{2p}}
		\end{equation*}
		which leads to the desired result by taking $f=g$.
	\end{proof}

\subsection{Bound on \texorpdfstring{$\Nrm{\op-\protect\widetilde{\op}}{\L^2}$}{auxiliary and Hartree} }\label{sec:op-opt}

	In Equation \eqref{eq:sqrt-op_optilde}, $\Nrm{\sqrt{\op}}{\L^\infty}$ and $\bNrm{\sqrt{\widetilde{\op}}}{\L^\infty}$ are bounded because $\op\in\L^\infty$. The main result of this section is Proposition~\ref{prop:sqrt-op-vs-opt} which gives an estimate for $\bNrm{\sqrt{\op}-\sqrt{\widetilde{\op}}}{\L^2}$ in terms of $\hbar$.
	
	To prove the proposition, we need to first establish some preliminary results. It will be useful to exchange the role of $x$ and $\xi$ to manipulate weights of the form $\weight{x}$ instead of weights of the form $\weight{\opp}$. The quantum analogue of exchanging the $x$ and $\xi$ variables is obtained by conjugation with the semiclassical Fourier transform $\ecF_h\varphi(\xi):= h^{\frac32}\,\widehat{\varphi}(h\xi)$. More precisely, defining $\op^\star := \ecF^{-1}_h\op\, \ecF_h$, then it holds
	\begin{equation}\label{eq:property_op_star}
		\op_f^\star  = \op_{f^\star} \quad \text{ with } \quad f^\star(x, \xi) = f(\xi, x).
	\end{equation}
	This exchange operation is a linear automorphism which preserves the $\L^p$ norms. Moreover, from the definition of a Fourier multiplier, $\weight{\opp}\ecF_h=\ecF_h\weight{x}$ and so $\weight{x}^\star = \weight{\opp}$ and $\weight{\opp}^\star = \weight{x}$.

	\begin{proof}[Proof of~\eqref{eq:property_op_star}]
		Computing the kernel of $\op_f^\star$ yields
		\begin{equation*}
			\op_f^\star(x,y) = \frac{1}{h^3}\int_{\RR^9} e^{2\pi i\, \(x\cdot x'-y'\cdot y\)/h}\,e^{-2\pi i\, \(y'-x'\)\cdot\xi}\, f\!\(\tfrac{x'+y'}{2}, h\,\xi\) \d x'\dd y'\dd \xi
		\end{equation*}
		Therefore, by the change of variables $\zeta = \frac{x'+y'}{2\,h}$ and $\eta = x'-y'$
		\begin{equation*}
			\op_f^\star(x,y) = \intdd e^{-2\pi i\, \(y-x\)\cdot\zeta/h} \, e^{2i\pi\eta\cdot\frac{x+y}{2\,h}}\,\ecF\!\(f(h\,\zeta, h\,\cdot)\)\!(\eta) \dd \eta\dd \zeta
		\end{equation*}
		and we deduce Equation~\eqref{eq:property_op_star} by the Fourier inversion theorem.
	\end{proof}
	
	We can now use the above defined operation to prove the following lemma.

	\begin{lem}\label{lem:Weyl_and_weight}
		Let $f\in H^1$ be a function on the phase space. Then the following estimates hold
		\begin{subequations}
			\begin{align}\label{eq:p-weight_remainder}
				\Nrm{\op_{\weight{\xi} f}-\op_{f}\weight{\opp}}{\L^2} &\le \tfrac{\hbar}{2}\Nrm{\grad_{x} f}{L^2},
				\\\label{eq:p-weight_remainder_2}
				\Nrm{\weight{\opp}\op_f \weight{\opp} -\op_{\weight{\xi}^2f}}{\L^2} &\le \tfrac{\hbar^2}{4}\Nrm{\Delta_x f}{L^2}.
			\end{align}
			Notice that the same inequality holds when we replace $\op_{f}\weight{\opp}$ by $\weight{\opp}\op_{f}$, which follows by taking the adjoint.
		\end{subequations}
	\end{lem}

	\begin{proof}
		Notice the kernel of the difference $\op_{\weight{x}f}-\op_{f}\weight{x}$ is given by 
		\begin{multline*}
			\intd e^{-2\pi i \(y-x\)\cdot \xi}\(\weight{\tfrac{x+y}{2}}-\weight{y}\)f\(\tfrac{x+y}{2}, h\xi\)\d \xi
			\\
			= \frac{i\hbar}{2}\frac{(\frac{x+y}{2})+ y}{\weight{\tfrac{x+y}{2}}+\weight{y}}\cdot \intd e^{-2\pi i \(y-x\)\cdot \xi}\, \grad_{\xi}f\(\tfrac{x+y}{2}, h\xi\)\dd \xi = \op_r
		\end{multline*}
		where $r \in L^2$ is defined by $r(z)=\tfrac{i\hbar}{2}\,(a_x(i\hbar\, \grad_\xi)\cdot\grad_{\xi}f)(z)$ where $a_x(i\hbar\, \grad_\xi)$ is a bounded, $x$-dependent Fourier multiplier of the $\xi$ variable associated to the function $a_x(\eta) = \frac{2\,x- \eta/2}{\weight{x}+\weight{x-\eta/2}}$. More explicitly
		\begin{equation*}
			r(z) = \frac{i\hbar}{2}\intdd e^{2\pi i\(\xi-\xi'\)\cdot \eta}\, a_x(h\eta)\cdot \grad_{\xi}f(x, \xi')\dd\xi'\dd\eta.
		\end{equation*}
		In particular, it follows that
		\begin{equation}\label{eq:x-weight_remainder}
			\Nrm{\op_{\weight{x} f}-\op_{f}\weight{x}}{\L^2} = \Nrm{\op_{r}}{\L^2} =  \Nrm{r}{L^2} \le \tfrac{\hbar}{2} \Nrm{\grad_{\xi}f}{L^2}.
		\end{equation}
		We now use Inequality~\eqref{eq:x-weight_remainder} together with the properties of the exchange operation~\eqref{eq:property_op_star} to get
		\begin{align*}
			\Nrm{\op_{\weight{\xi} f}-\op_{f}\weight{\opp}}{\L^2} &= \Nrm{\op_{\weight{\xi} f}^\star-\op_{f}^\star\weight{x}}{\L^2} = \Nrm{\op_{\weight{x} f^\star}-\op_{f^\star}\weight{x}}{\L^2}
			\\
			&\le \hbar \Nrm{\grad_{\xi}f^\star}{L^2} = \hbar \Nrm{(\grad_{x}f)^\star}{L^2} = \hbar \Nrm{\grad_{x}f}{L^2},
		\end{align*}
		which completes the proof of Inequality~\eqref{eq:p-weight_remainder}. To get the second inequality, a similar computation gives that
		\begin{equation*}
			\weight{x}\op_f \weight{x} -\op_{\weight{x}^2f} = \op_{r_2}
		\end{equation*}
		with $r_2(z)=\tfrac{-\hbar^2}{4}\,(b_x(i\hbar\, \grad_\xi)\cdot\Delta_{\xi}f)(z)$, where $b_x$ is the bounded function given by
		\begin{equation*}
			b_x(2\eta) = \tfrac{2\n{x}^2-2-\n{\eta}^2}{\weight{x}^2+\weight{x+\eta}\weight{x-\eta}} = \tfrac{\n{x}^2-2+\(x+\eta\)\(x-\eta\)}{\weight{x}^2+\weight{x+\eta}\weight{x-\eta}},
		\end{equation*}
		and the same reasoning leads to Inequality~\eqref{eq:p-weight_remainder_2}.
	\end{proof}

	Another simple observation is that the commutator between convolution with the Gaussian function $g_h(z) = \(\pi\hbar\)^{-3}e^{-\n{z}^2/\hbar}$ and multiplication by a weight is also of order $\hbar$.
	\begin{lem}\label{lem:p-weight_remainder_gaussian}
		Let $h\in(0,1)$. Then there exists $C>0$ such that for any $p\in[1,\infty]$ and any function of the phase space $f \in W^{1,p}$ it holds
		\begin{subequations}
			\begin{align}\label{eq:p-weight_remainder_gaussian}
				\Nrm{\com{\weight{\xi},\sim}f}{W^{2,p}} \leq C\,\hbar \Nrm{f}{W^{3,p}}
				\\\label{eq:p-weight_remainder_gaussian_2}
				\Nrm{\com{\langle\xi\rangle^2,\sim}f}{L^p} \leq C\,\hbar \Nrm{f}{W^{1,p}_1}
			\end{align}
		\end{subequations}
		where $\com{m,\sim} f := m\(g_h*f\) -  \(g_h*(mf)\)$.
	\end{lem}
	
	\begin{proof}
		For any $z=(x,\xi)\in\Rdd$ and any general $m(\xi)$, it holds 
		\begin{equation*}
			\com{m,\sim} f(z) = \intd g_h(z') \(m(\xi)-m(\xi-\xi')\) f(z-z')\dd z'.
		\end{equation*}
		By a second order Taylor expansion, this can be written as
		\begin{equation*}
			\com{m,\sim} f = \intd g_h(z')\, \xi'\cdot\(\nabla m(\xi)- \int_0^1\theta\, \xi'\cdot \nabla^2m(\xi-\theta\xi')\dd\theta \) f(z-z')\dd z' = I_1 - I_2.
		\end{equation*}
		Next, using the fact that that $g_h$ is even, one can replace $f(z-z')$ by $f(z-z')-f(z)$ in $I_1$, leading to
		\begin{equation*}
			I_1 = \nabla m(\xi)\cdot \int_0^1\!\!\!\intd \xi'\,g_h(z')\, \xi' \cdot\,\Dv f(z-\theta \xi')\dd z' \dd\theta.
		\end{equation*}
		By Minkowski's inequality, we obtain
		\begin{align*}
			\Nrm{I_1}{L^p} \le \Nrm{\nabla m}{L^\infty} \,\bNrm{\!\n{\xi}^2g_h}{L^1} \Nrm{\Dv f}{L^p} = \tfrac{3}{2}\,\hbar\,\Nrm{\nabla m}{L^\infty}\Nrm{\Dv f}{L^p}.
		\end{align*}
		Similarly, $I_2$ is bounded by
		\begin{align*}
			\Nrm{I_2}{L^p} \le \frac{1}{2}\Nrm{\nabla^2 m}{L^\infty}  \,\bNrm{\!\n{\xi}^2g_h}{L^1}\Nrm{f}{L^p} = \tfrac{3}{4}\,\hbar\,\Nrm{\nabla^2 m}{L^\infty}\Nrm{f}{L^p}.
		\end{align*}
		In particular, if $m = \weight{\xi}$, then it yields $\Nrm{\com{m,\sim} f}{L^p} \leq 3\,\hbar \Nrm{f}{W^{1,p}}$. Noticing that $\Dx^2\com{m,\sim} f = \com{m,\sim} \Dx^2 f$ also yields $\Nrm{\Dx^2\com{m,\sim} f}{L^p} \leq 3\,\hbar \Nrm{\Dx^2 f}{W^{1,p}}$. Finally, noticing that
		\begin{equation*}
			\Dv^2\com{m,\sim} f = \com{m,\sim} \Dv^2 f + 2\com{\nabla m,\sim} \Dv f + \com{\nabla^2 m,\sim} f
		\end{equation*}
		and using the above estimates with $m$ replaced by $\nabla m$ and $\nabla^2 m$ leads to 
		\begin{equation*}
			\Nrm{\Dv^2\com{m,\sim} f}{L^p} \leq 36 \,\hbar \Nrm{f}{H^3}	
		\end{equation*}
		and so to Inequality~\eqref{eq:p-weight_remainder_gaussian}. To get Inequality~\eqref{eq:p-weight_remainder_gaussian_2}, notice that $\com{\weight{\xi}^2,\sim}f = \com{\n{\xi}^2,\sim}f$ and write
		\begin{equation*}
			\com{\n{\xi}^2,\sim}f = \intdd g_h(z')\(\xi'\)^{\otimes 2} : \(f(z-z')\,\Id + 2\, \xi\otimes \int_0^1 \Dv f(z-\theta \,\xi')\dd\theta\)\d z'
		\end{equation*}
		so that 
		\begin{align*}
			\Nrm{\com{\n{\xi}^2,\sim}f}{L^p}\le \tfrac{3\,\hbar}{2} \(\Nrm{f}{L^p} + 2 \Nrm{\n{\xi}\Dv f}{L^p}\) + \tfrac{16\,\hbar^{\frac32}}{\sqrt{\pi}} \Nrm{\Dv f}{L^p},
		\end{align*}
		and this concludes the proof.
	\end{proof}

	\begin{lem}\label{lem:init_diff_diag}
		Let $h\in(0,1)$. Then there exists $C>0$ such that for any function $g$ on the phase space
		\begin{equation*}
			\Nrm{\weight{\opp}\(\widetilde{\op}_{g}^2 - \widetilde{\op}_{g^2}\) \weight{\opp}}{\L^2} \leq C\,\hbar \(\Nrm{g}{W^{1,4}_1\cap H^3}^2 + \Nrm{g^2}{H^1_1\cap H^2}\).
		\end{equation*}
	\end{lem}

	\begin{proof}
		Let $\opm := \weight{\opp}$ and $m(\xi) = \weight{\xi}$. By the definition of the absolute value for operators, by the proof of Lemma~\ref{lem:Weyl_and_weight} and by Lemma~\ref{lem:p-weight_remainder_gaussian}, we can write
		\begin{equation*}
			\opm\,\widetilde{\op}_{g}^2\,\opm = \n{\widetilde{\op}_{g} \opm}^2 = \n{\op_{\tilde{g}} \opm}^2 = \n{\op_{m\,\tilde{g}} + \op_r}^2 = \n{\widetilde{\op}_{m\,g} + \op_{r_1}}^2
		\end{equation*}
		where $r_1 = \com{m,\sim}g+r$, and 
		\begin{equation*}
			\opm\,\widetilde{\op}_{g^2}\,\opm = \op_{m^2\widetilde{g^2}} + \op_{r_2} = \widetilde{\op}_{m^2\,g^2} + \op_{r_0}
		\end{equation*}
		where $r_0 = \com{m^2,\sim}g^2+r_2$. It follows that 
		\begin{multline}\label{est:difference_toeplitz_square}
			\Nrm{\opm\Big(\widetilde{\op}_{g}^2 - \widetilde{\op}_{g^2}\Big)\opm}{\L^2} = \Nrm{\n{\widetilde{\op}_{m\,g} +\op_{r_1}}^2 - \widetilde{\op}_{m^2\,g^2} - \op_{r_0}}{\L^2}
			\\
			\leq \Nrm{\widetilde{\op}_{m\,g}^2 - \widetilde{\op}_{m^2\,g^2}}{\L^2} + 2 \Nrm{\widetilde{\op}_{m\,g}\, \op_{r_1}}{\L^2}  + \Nrm{\op_{r_1}}{\L^4}^2  + \Nrm{\op_{r_0}}{\L^2}.
		\end{multline}
		The first term on the right-hand side of Inequality \eqref{est:difference_toeplitz_square} is bounded by Lemma~\ref{lem:commut_square_wick} with $p=2$, leading to
		\begin{equation*}
			\Nrm{\widetilde{\op}_{m\,g}^2 - \widetilde{\op}_{m^2\,g^2}}{\L^2} \leq 48\,\hbar \Nrm{\nabla(m\,g)}{L^4}^2.
		\end{equation*}
		The second term can be bounded by H\"older's inequality and Inequality~\eqref{eq:Wick_Schatten} as follows
		\begin{equation*}
			\Nrm{\widetilde{\op}_{m\,g} \, \op_{r_1}}{\L^2} \leq \Nrm{\widetilde{\op}_{m\,g}}{\L^4} \Nrm{\op_{r_1}}{\L^4} \leq  C\,\hbar \Nrm{m\,g}{L^4} \Nrm{g}{H^3}.
		\end{equation*}
		The third term can be bounded by
		\begin{equation*}
			\Nrm{\op_{r_1}}{\L^4} \leq \Nrm{\op_{r_1} - \widetilde{\op}_{r_1}}{\L^4} +\Nrm{\widetilde{\op}_{r_1}}{\L^4}
		\end{equation*}
		and then noticing that the Schatten-$4$ norm is controlled by the Hilbert--Schmidt norm, one obtains $\Nrm{\op_{r_1} - \widetilde{\op}_{r_1}}{\L^4} \leq h^{-\frac34} \Nrm{\op_{r_1} - \widetilde{\op}_{r_1}}{\L^2} \leq h^{\frac14} \Nrm{\nabla^2 r_1}{L^2}$, where we used Inequality~\eqref{eq:Wick_Weyl_error} in the last inequality. Hence we deduce that
		\begin{equation*}
			\Nrm{\op_{r_1}}{\L^4} \leq \Nrm{r_1}{L^4} + h^{\frac14} \Nrm{\nabla^2 r_1}{L^2} \leq C\,\hbar \Nrm{g}{H^3}.
		\end{equation*}
		Finally, the last term  can be bounded by
		\begin{align*}
			\Nrm{\op_{r_0}}{\L^2} \le C\, \hbar \(\Nrm{g^2}{H^1_1} + \hbar \Nrm{\Delta_x g^2}{L^2}\).
		\end{align*}
		Combining all the inequalities leads to the desired result.
	\end{proof}

	\begin{lem}\label{lemma:diag-opt-opf}
		Let $\op_f$ be a solution to Equation~\eqref{eq:Weyl-Vlasov} and $\widetilde{\op}$ be a solution to Equation~\eqref{eq:linear-Hartree}. Then 
		\begin{equation*}
			\Nrm{\diag(\widetilde{\op}-\op_f)}{L^2_x}\leq C\, \hbar \(C^\init+ \Nrm{\rho_f}{W^{1,\infty}_x} \Nrm{\op_f}{\cW^{2,2}(\weight{\opp}^2)}\),
		\end{equation*}
		where $C^\init$ is given by Equation~\eqref{eq:C_init}.
	\end{lem}
	
	\begin{proof}
		Let $\opm = \weight{\opp}$. By the proof of \cite[Proposition~6.4]{chong_manybody_2021} and the equation verified by $\widetilde{\op}-\op_f$, we obtain
		\begin{align*}
			\Nrm{\diag(\widetilde{\op}-\op_f)}{L^2_x} &\leq C\Nrm{\opm\,(\widetilde{\op}-\op_f)\,\opm}{\L^2}
			\\
			&\leq C\Nrm{\opm\,(\widetilde{\op}^\init-\op_f^\init)\,\opm}{\L^2}+\frac{C}{\hbar}\int_0^t \Nrm{B_f(\op_f)\,\opm^2}{\L^2} \d s.
		\end{align*}
		With our choice of the initial datum $\widetilde{\op}^\init = \widetilde{\op}^2_{\sqrt{f^\init}}$, the first term on the right-hand side is bounded in Lemma~\ref{lem:init_diff_diag}. As for the second term, we compute explicitly the $L^2$ norm of the kernel of the operator $B_f(\op_f)\,\opm^2$. Recall the definition of $B_f(\op)$ given by Formula~\eqref{def:B_operator} and define
		\begin{equation*}
			\delta^2 V(x,y) := V_{f}(x) - V_{f}(y) - \(x-y\)\cdot\nabla V_{f}\Big(\frac{x+y}{2}\Big)
		\end{equation*}
		so that the integral kernel of $B$ can be written $B_f(\op)(x,y) = \delta^2 V(x,y)\,\op(x,y)$. Then by the chain rule, the integral kernel of the operator $B_f(\op_f)\n{\opp}^2$ is given by
		\begin{align*}
			\big(B_f(\op_f)\n{\opp}^2\big)(x,y)
			&= -\hbar^2\, \delta^2V(x,y)\, \lapl_y\op_f(x,y) -\hbar^2 \lapl_y\!\(\delta^2V(x,y)\) \op_f(x,y)
			\\
			&\qquad- 2\,\hbar^2\,  \grad_y\delta^2V(x, y)\cdot \grad_y\op_f(x,y) =: J_1+J_2+J_3.
		\end{align*}
		The $J_1$ term is simply $B_f(\op_f\n{\opp}^2)$ which can be bounded in $\L^2$ by Inequality~\eqref{eq:B_bound}, leading to
		\begin{align*}
			\Nrm{B_f(\op_f\n{\opp}^2)}{\L^2} \le C \, \hbar^2 \Nrm{\rho_f}{W^{1,\infty}_x} \Nrm{\Dhv^2\!\(\op_f\n{\opp}^2\)}{\L^2}.
		\end{align*}
		For the $J_2$ term, since $V$ is the solution of the Poisson equation, we have that 
		\begin{align*}
			-\hbar^2\,\lapl_y \delta^2V(x, y) = \pm\hbar^2 \Big(
				\rho_{f}\Big(\frac{x+y}{2}\Big)-\rho_{f}(y)\Big) - \tfrac{\pm \hbar^2}{4}\, \grad \rho_f\Big(\frac{x+y}{2}\Big)\cdot(x-y)
		\end{align*}
		which means
		\begin{align*}
			J_2(x,y) = \pm\hbar^2 \Big(\rho_{f}\Big(\frac{x+y}{2}\Big)-\rho_{f}(y)\Big)\op_{f}(x, y)
			\pm\frac{\hbar^3}{4i} \grad \rho_f\Big(\frac{x+y}{2}\Big)\cdot \Dhv \op_{f}(x,y).
		\end{align*}
		Then we deduce that
		\begin{equation*}
			\Nrm{J_2}{\L^2} \le 2\,\hbar^2 \Nrm{\rho_f}{L^\infty_x}\Nrm{\op_f}{\L^2}+\tfrac{\hbar^3}{4} \Nrm{\grad \rho_f}{L^\infty_x}\Nrm{\Dhv\op_f}{\L^2}.
		\end{equation*}
		Lastly, for the $J_3$ term, observe that 
		\begin{align*}
			-2\,\hbar^2\, \nabla_y \delta^2 V(x, y)= 2\,\hbar^2\(E_f(\tfrac{x+y}{2})-E_f(y)\)+\hbar^2 \(x-y\)\cdot \grad E_f\(\tfrac{x+y}{2}\)
		\end{align*}
		which means 
		\begin{multline*}
			J_3(x,y) = i\hbar^3 \int^1_0 \nabla E_f(y+t\,\tfrac{x-y}{2}) : \nabla_y\Dhv\op_f(x,y) \dd t
			\\
			+ i\hbar^3 \grad E_f\(\tfrac{x+y}{2}\) : \nabla_y\Dhv\op_f(x,y).
		\end{multline*}
		Hence it follows that 
		\begin{align*}
			\Nrm{J_3}{\L^2} \le C\,\hbar^2 \Nrm{\grad E_f}{L^\infty_x}\Nrm{\Dhv\op_f\, \opp}{\L^2}.
		\end{align*}
		This completes our proof of the lemma. 
	\end{proof}
	
	\begin{prop}\label{prop:sqrt-op-vs-opt}
		Let $\op$ be a solution to the Hartree equation \eqref{eq:Hartree} with initial datum $\op^\init$ and $\widetilde{\op}$ be a solution to \eqref{eq:linear-Hartree} with initial datum $\widetilde{\op}^\init$. We denote by $v_1 := \sqrt{\op}$, $\tilde{v} := \sqrt{\widetilde{\op}}$. Then there exists $\lambda(t)\in C^0(\R_+,\R_+)$ such that 
		\begin{equation*}
		\Nrm{v_1-\tilde{v}}{\L^2}^2\leq \Nrm{v_1^\init - \tilde{v}^\init}{\L^2}^2 e^{2\,\Lambda(t)} + \hbar^2 \int_0^t c(s)^2 \,e^{2\(\Lambda(t)-\Lambda(s)\)}\dd s
		\end{equation*}
		where $\Lambda(t) = C \int_0^t \lambda(s)\dd s$ with $\lambda$ given by~\eqref{eq:lambda}
		and $c(t)$ is given by
		\begin{equation*}
			c(t) = C\, \Nrm{\Dhv \tilde{v}}{\cW^{1,2}} \(C^\init + \Nrm{\rho_f}{W^{1,\infty}_x} \Nrm{\op_f}{\cW^{2,2}(\weight{\opp}^2)}\)
		\end{equation*}
		with $C^\init$ given in Equation~\eqref{eq:C_init}.
	\end{prop}
		
	\begin{proof}
		Define $v:=v_1-\tilde{v}$. Observe that $v_1$ solves
		\begin{equation*}
			i\hbar\,\dpt v_1=\com{H_1,v_1},\quad\quad H_1:=H_{v_1^2}
		\end{equation*}
		and $\tilde{v}$ solves
		\begin{equation*}
			i\hbar\,\dpt \tilde{v}=\com{H_f,\tilde{v}}.
		\end{equation*}
		Now we proceed by mimicking the proof of Theorem~\ref{thm:Q-uniqueness}. By direct computation, we have that
		\begin{equation*}
			i\hbar\,\dpt v = \com{H_1,v}+\bcom{V_{\op}-V_{\op_f},\tilde{v}}
		\end{equation*}
		which implies 
		\begin{equation}\label{eq:2norm-v}
			\frac{1}{4\pi}\,\dt\Nrm{v}{\L^2}^2 = h^2 \Tr{\bcom{V_{v_1^2}-V_{\op_f},\tilde{v}}v}.
		\end{equation}
		Now adding and subtracting $h^2 \Tr{\com{V_{\tilde{\op}},\tilde{v}}v}$ to the right-hand side of Equation~\eqref{eq:2norm-v} yields
		\begin{equation}\label{eq:L2-estimate}
			\begin{split}
			\frac{1}{4\pi}\,\frac{\d}{\d t}\Nrm{v}{\L^2}^2 &= h^2 \Tr{\com{V_{\op}-V_{\tilde{\op}},\tilde{v}}v} + h^2 \Tr{\bcom{V_{\tilde{\op}}-V_{\op_f},\tilde{v}}v}
			\\
			&= h^2 \Tr{\bcom{V_{v_1^2-\tilde{v}^2},\tilde{v}}v} + h^2 \Tr{\bcom{V_{\tilde{\op}-\op_f},\tilde{v}}v}
			=: J_1+J_2.
			\end{split}
		\end{equation}
		We bound $J_1$ in the same manner as in Theorem~\ref{thm:Q-uniqueness} and get that 
		\begin{equation*}
			J_1 \leq h^2 \Tr{\com{V_{v^2},\tilde{v}}v} + 2\, h^2 \Tr{\com{V_{\tilde{v}v},\tilde{v}}v} =: J_{1,1}+2\,J_{1,2},
		\end{equation*}
		where the two terms are readily bounded as follows
		\begin{align*}
			J_{1,1} &\leq C\Nrm{\Dhv\tilde{v}\,\opm}{\L^{3\pm\eps}}\Nrm{v}{\L^\infty}\Nrm{v}{\L^2}^2,
			\\
			J_{1,2} &\leq C\Nrm{\com{V_{v\tilde{v}},\tilde{v}}}{\L^2}\Nrm{v}{\L^2}.
		\end{align*}
		The right-hand side of $J_{1,1}$ is controlled by $\Nrm{v}{\L^2}^2$ since $\Nrm{\Dhv \tilde{v}\,\opm}{\L^{3\pm\eps}}$ is bounded thanks to Proposition~\ref{prop:propagation_regularity} and $\Nrm{v}{\L^\infty} \leq \Nrm{v_1}{\L^\infty} + \Nrm{\tilde{v}}{\L^\infty}$, which are both bounded. By Lemma~\ref{lem:commutator_estimate_HS}, similarly as in the proof of Theorem~\ref{thm:Q-uniqueness}, we get
		\begin{equation*}
			J_{1,2} \leq C \Nrm{\Dhv \tilde{v}}{\cW^{1,2}} \Nrm{\rho_f}{L^\infty_x}^{\frac12} \Nrm{v}{\L^2}^2
		\end{equation*}
		where $\rho_f\in L^\infty(\R^3)$ and $\Nrm{\Dhv \tilde{v}}{\cW^{1,2}}$  is bounded in time thanks to Proposition~\ref{prop:propagation_regularity}. We are now left with the bound on $J_2$. By H\"older's inequality for Schatten norms and Young's inequality for products we obtain
		\begin{equation*}
			J_2 = h^2 \Tr{\bcom{V_{\tilde{\op}-\op_f},\tilde{v}}v} \leq \bNrm{\tfrac{1}{h}\bcom{V_{\tilde{\op}-\op_f},\tilde{v}}}{\L^2}^2+ \Nrm{v}{\L^2}^2.
		\end{equation*}
		We aim at showing that $\tfrac{1}{h}\bNrm{\bcom{V_{\tilde{\op}-\op_f},\tilde{v}}}{\L^2}$ is small. To this end, apply Lemma~\ref{lem:commutator_estimate_HS} to get
		\begin{equation}\label{eq:V-opt-opf}
			\tfrac{1}{h}\bNrm{\bcom{V_{\tilde{\op}-\op_f},\tilde{v}}}{\L^2} \leq C \Nrm{\Dhv \tilde{v}}{\cW^{1,2}} \Nrm{\diag(\widetilde{\op}-\op_f)}{L^2_x}.
		\end{equation}
		The term $\Nrm{\diag(\widetilde{\op}-\op_f)}{L^2_x}$ in Equation~\eqref{eq:V-opt-opf} is controlled in Lemma~\ref{lemma:diag-opt-opf}, leading to the statement of the proposition.
	\end{proof}
	
\subsection{Proof of the main theorems}
	 
	Now that we have a bound on $\sqrt{\op}-\sqrt{\widetilde{\op}}$ in $\L^2$ by the above proposition, we are ready to finish the proof of the semiclassical limit.
	\begin{proof}[Proof of Theorem~\ref{thm:main} and Theorem~\ref{thm:main_2}]
		As explained in the beginning of the section, by Lemma~\ref{lem:commut_square_wick} with $p=2$, we get that
		\begin{align}\label{eq:initial-state-f}
			\Nrm{\widetilde{\op}_{f^\init} - \widetilde{\op}^\init}{\L^2} \leq 48\,\hbar \Nrm{\nabla \sqrt{f^\init}}{L^4}^2,
		\end{align}
		and so, thanks to Inequality~\eqref{eq:optilde-opf} and the fact that $\Nrm{\nabla E_f}{L^\infty_x}$ can be controlled by $\Nrm{\rho_f}{W^{1,\infty}_x}$, the error of positivity is given by
		\begin{equation}\label{eq:bound-rho-tilde-rhof}
			\Nrm{\widetilde{\op}-\op_f}{\L^2} \leq C\,\hbar\( \Nrm{\nabla \sqrt{f^\init}}{L^4}^2 + \int_0^t \Nrm{\rho_f}{W^{1,\infty}_x} \Nrm{\nabla_{\xi}^2 f}{L^2}\).
		\end{equation}
		On the other side, the nonlinear stability error~\eqref{eq:sqrt-op_optilde} is controlled, using Proposition~\ref{prop:sqrt-op-vs-opt} and recalling that $\tilde{\op}^\init = \widetilde{\op}_{\sqrt{f^\init}}^2$, by
		\begin{equation}\label{eq:bound-eq:bound-rho-rho-tilde}
			\Nrm{\op-\widetilde{\op}}{\L^2} \leq \cC_{\infty}^\frac{1}{2}\(\Nrm{\sqrt{\op^\init} - \widetilde{\op}_{\sqrt{f^\init}}}{\L^2}  + \hbar \Nrm{c}{L^2([0,t])}\) e^{\Lambda(t)}.
		\end{equation}
		Summing up \eqref{eq:bound-rho-tilde-rhof} and \eqref{eq:bound-eq:bound-rho-rho-tilde} and using the estimate on the initial state \eqref{eq:initial-state-f} prove Theorem~\ref{thm:main_2}.
		
		Moreover, we can control the difference of the square roots using Inequality~\eqref{eq:Wick_Weyl_error} which tells us that the Wick and Weyl quantization are close since $\sqrt{f^\init}$ is sufficiently regular, and then the fact that the Wigner transform is an isometry from $\L^2$ to $L^2$ to get
		\begin{equation*}
			\Nrm{\sqrt{\op^\init} - \widetilde{\op}_{\sqrt{f^\init}}}{\L^2} \leq \Nrm{\sqrt{f^\init} - f_{\sqrt{\op^\init}}}{L^2} + \tfrac{3\,\hbar}{2} \Nrm{\nabla^2 \sqrt{f^\init}}{L^2}
		\end{equation*}
		which leads to our main inequality~\eqref{eq:Wigner_main_estimate}.
	\end{proof}

	\subsection{Propagation of regularity}
	
	In this section, we prove the following proposition about the semiclassical propagation of regularity for $\tilde{v}$ assuming sufficient regularity on the solution to the Vlasov--Poisson equation. 
	\begin{prop}\label{prop:propagation_regularity}
		Let $h\in(0,1)$, $q\in[1,\infty)$, $n\in\N$, $f$ be a smooth solution of the Vlasov--Poisson equation~\eqref{eq:Vlasov} and $\tilde{v}$ be the solution to the linear equation 
		\begin{equation*}
			i \hbar\,\dpt \tilde{v} = \com{H_f, \tilde{v}},
		\end{equation*}
		where $H_f = -\tfrac{\hbar^2}{2}\lapl + V_f$. Define $\opm = \weight{\opp}^{2n}$. Then for any $\eps\in(0,1)$, there exists $C_{q,n}$ depending only on $q$, $n$ and $\eps$, such that for every $t>0$ it holds
		\begin{equation*}
			\Nrm{\tilde{v}(t)}{\cW^k(\opm)} \le \Nrm{\tilde{v}^\init}{\cW^k(\opm)} \exp\! \(C_{q,n} \int^t_0 \Nrm{\rho_f(s,\cdot)}{W_x^{2n,3\pm\eps}} \d s\).
		\end{equation*}
	\end{prop}
	
	\begin{remark}
		To obtain the result of Theorem~\ref{thm:main}, we need bounds on the norm of $\tilde{v}$ in $\cW^{2,2}\cap \cW^{1,3\pm\eps}(\weight{\opp}^{2n})$ with $2n>2$. In particular, by the inequalities in~\cite{lafleche_strong_2021}, these norms are bounded for the initial data $\tilde{v}^\init = \tilde{\op}_{\sqrt{f^\init}}$ if 
		\begin{equation*}
			\Nrm{\sqrt{f^\init}}{H^2\cap W^{1,3\pm\eps}_{2n}} +
			\hbar^2 \Nrm{\sqrt{f^\init}}{W^{4,3\pm\eps}}
		\end{equation*}
		is bounded uniformly in $\hbar$.
	\end{remark}
	 
	\begin{proof}
		By the Jacobi identity, we arrive at the equation
		\begin{equation}\label{eq:qgrad_x-equation}
		    \bd_t \Dhx^k \tilde{v} = \tfrac{1}{i\hbar} \com{H_f, \Dhx^{k} \tilde{v}} -\sum^k_{j=1}\binom{k}{j}\tfrac{1}{i\hbar}\com{\Dhx^{j-1} E_f, \Dhx^{k-j} \tilde{v}}.
		\end{equation}
		Likewise, using the fact that $\Dhv^j H_f = 0$ for $j\ge 3$, we have that 
		\begin{equation*}
		    \bd_t \Dhv^k \tilde{v} = \tfrac{1}{i\hbar} \com{H_f, \Dhv^k \tilde{v}} - k\,\Dhv^{k-1}\Dhx \tilde{v}
		\end{equation*}
		and
		\begin{multline*}
		    \bd_t \Dhx^\ell \Dhv^{k-\ell} \tilde{v} = \tfrac{1}{i\hbar} \com{H_f, \Dhx^\ell \Dhv^{k-\ell} \tilde{v}} - \(k-\ell\)\Dhv^{k-\ell-1}\Dhx^{\ell+1} \tilde{v}
		    \\
		    -\sum^\ell_{j=1}\binom{\ell}{j}\tfrac{1}{i\hbar} \com{\Dhx^{j-1} E_f, \Dhx^{\ell-j} \Dhv^{k-\ell} \tilde{v}}
		\end{multline*}
		where $\ell <k$. Let us focus on propagating moments for $\Dhx^k \tilde{v}$ since Equation~\eqref{eq:qgrad_x-equation} only depends on $\Dhx^\ell \tilde{v}$ for $\ell \le k$. Define a weight equivalent to $\opm$ by $\widetilde \opm= 1+\sum^3_{j=1}\bp^{2n}_\ii$ as in~\cite[Lemma~6.3]{chong_manybody_2021}. Then, by~\cite[Lemma~6.2]{chong_manybody_2021}, we arrive at the estimate 
		\begin{equation*}
		    \dt\nrm{\Dhx^k \tilde{v}\, \widetilde \opm}_{q} \le \frac{1}{\hbar }\nrm{\com{V_f, \widetilde \opm} \Dhx^k \tilde{v}}_{q}+\frac{1}{\hbar}\sum^k_{j=1}\binom{k}{j}\nrm{\com{\Dhx^{j-1}E_f, \Dhx^{k-j} \tilde{v}} \widetilde\opm}_{q}
		\end{equation*}
		for $q\ge 2$. Let us estimate the first term. To simplify the notation, we write $V=V_f$, $E=E_f$, and $\opmu = \Dhx^{k} \tilde{v}$. Recall the identity $\com{V, \bp} = i\hbar\, E$. Then it follows 
		\begin{equation*}
		    \frac{1}{i\hbar} \com{V, \bp^{2n}_{\ii}} = \sum^{2n-1}_{k=0} \bp^k_\ii\, E_\ii\, \bp_\ii^{2n-1-k} =\sum^{2n-1}_{k=0} \sum^k_{\ell = 0} \binom{k}{\ell} g_\ell\, \bp_\ii^{2n-1-\ell}
		\end{equation*}
		where $g_\ell = \(-i\hbar\)^\ell \bd_\ii^\ell E_\ii$. Using the above identity yields the estimate 
		\begin{align*}
		    \frac{1}{\hbar} \Nrm{\com{V, \bp_\ii^{2n}} \opmu}{\cL^q} 
		    &\le \sum^{2n-1}_{\ell =0} \binom{2n}{\ell+1} \Nrm{g_\ell}{L^\infty} \Nrm{\opmu\,\opm}{\cL^q}.
		\end{align*}
		Notice that 
		\begin{equation*}
		    \Nrm{g_\ell}{L^\infty} = \hbar^\ell \Nrm{\grad K * \bd_\ii^{\ell} \rho_f}{L^\infty} \leq C \,\hbar^\ell \Nrm{\bd_\ii^{\ell} \rho_f}{L^{3\pm\eps}}
		\end{equation*}
		for $\eps\in(0,1)$, then it follows that 
		\begin{equation*}
		    \frac{1}{\hbar} \Nrm{\com{V, \bp_\ii^{2n}} \opmu}{\cL^q} \leq C_{f}(t) \(\tfrac{(1+\hbar)^{2n}-1}{\hbar}\) \Nrm{\opmu\,\opm}{\cL^q}.
		\end{equation*}
		Let us now estimate the second term. Write $\opmu_{k-j} = \Dhx^{k-j} \tilde{v}$. Using the fact that 
		\begin{align*}
		    \com{\Dhx^{j-1}E, \opmu_{k-j}} \bp_\ii^{2n} = \com{\Dhx^{j-1}E, \opmu_{k-j}\, \bp_\ii^{2n}} - \opmu_{k-j} \com{\Dhx^{j-1}E, \bp_\ii^{2n}}
		\end{align*}
		then it follows 
		\begin{multline*}
		    \frac{1}{\hbar}\nrm{\com{\Dhx^{j-1}E, \opmu_{k-j}} \bp_\ii^{2n}}_{\cL^q}
		    \\
		    \le \frac{1}{\hbar}\nrm{\com{\Dhx^{j-1}E, \opmu_{k-j}\, \bp_\ii^{2n}}}_{\cL^q}+\frac{1}{\hbar}\nrm{\opmu_{k-j}\com{\Dhx^{j-1}E,  \bp_\ii^{2n}}}_{\cL^q}=:I_1+I_2.
		\end{multline*}
		By \cite[Proposition~6.5]{chong_manybody_2021}, we have that 
		\begin{equation*}
		    I_1 \le C\Nrm{\grad^{j-1}\rho_f}{W^{1,r}}\Nrm{\Dhv\opmu_{k-j}\, \bp_\ii^{2n}}{\cL^q}\quad \text{ with } \quad \frac{1}{r}=\frac{1}{2}-\frac{1}{q} < \frac{1}{3}
		\end{equation*}
		and $k>j$. Term $I_2$ is handled in the same manner as in the case of $\tfrac{1}{i\hbar}\com{V, \bp_\ii^{2n}} \opmu$, which yields 
		\begin{equation*}
		    I_2 \le C_q \sup_{1\le \ell\le 2n} \nrm{\bd_\ii^{\ell} \rho_f}_{L^{3\pm\eps}} \(\tfrac{(1+\hbar)^{2n}-1}{\hbar}\) \Nrm{\opmu\,\opm}{\cL^q}.
		\end{equation*}
		Hence it follows 
		\begin{subequations}\label{est:v2_propagation_of_regularity}
		\begin{equation}
			\dt \Nrm{\Dhx^k \tilde{v}\, \widetilde \opm}{q} \le C_{q, f, n}(t) \(\Nrm{\Dhx^k \tilde{v}\, \widetilde \opm}{q} + \sum^k_{j=1}\binom{k}{j} \Nrm{\Dhv\Dhx^{k-j} \tilde{v}\, \widetilde \opm}{q} \).
		\end{equation}
		In fact, recycling the above argument yields the estimates 
		\begin{equation}
		    \dt \Nrm{\Dhv^k \tilde{v}\, \widetilde \opm}{q} \leq\, C_{q, f, n}(t)\( \nrm{\Dhv^k \tilde{v}\, \widetilde \opm}_{q}+\nrm{\Dhv^{k-1} \Dhx \tilde{v}\, \widetilde \opm}_{q} \)
		\end{equation}
		and 
		\begin{multline}
			\dt\Nrm{\Dhx^\ell \Dhv^{k-\ell} \tilde{v}\, \widetilde \opm}{q}
			\le C_{q, f, n}(t) \bigg( \Nrm{\Dhx^\ell \Dhv^{k-\ell} \tilde{v}\, \widetilde \opm}{q}
			\\
		    + (k-\ell) \Nrm{\Dhv^{k-1-\ell} \Dhx^{\ell+1} \tilde{v}\, \widetilde \opm}{q} + \sum^\ell_{j=1} \binom{\ell}{j} \Nrm{\Dhx^{\ell-j} \Dhv^{k-\ell} \tilde{v}\, \widetilde \opm}{q} \bigg).
		\end{multline}
		\end{subequations}
		Combining the above estimates~\eqref{est:v2_propagation_of_regularity} gives us a bound of the form
		\begin{equation*}
		    \dt \Nrm{\tilde{v}(t)}{\cW^k(\opm)} \le C_{q, f, n}'(t)  \Nrm{\tilde{v}(t)}{\cW^k(\opm)}
		\end{equation*}
		which then by Gr\"onwall's lemma leads to the result.
	\end{proof}
	\medskip
	
	{\bf Acknowledgments.} J.C. was supported by the NSF through the RTG grant DMS- RTG 184031. C.S. acknowledges the NCCR SwissMAP and the support of the SNSF through the Eccellenza project PCEFP2\_181153.


\bibliographystyle{abbrv} 
\bibliography{Vlasov}

\begin{thebibliography}{10}

\bibitem{amour_semiclassical_2013}
L.~Amour, M.~Khodja, and J.~Nourrigat.
\newblock The {{Semiclassical Limit}} of the {{Time Dependent
  Hartree}}\textendash{{Fock Equation}}: The {{Weyl Symbol}} of the
  {{Solution}}.
\newblock {\em Analysis {$\&$} PDE}, 6(7):1649--1674, 2013.

\bibitem{athanassoulis_strong_2011}
A.~Athanassoulis, T.~Paul, F.~Pezzotti, and M.~Pulvirenti.
\newblock Strong {{Semiclassical Approximation}} of {{Wigner Functions}} for
  the {{Hartree Dynamics}}.
\newblock {\em Rendiconti Lincei - Matematica e Applicazioni}, 22(4):525--552,
  2011.

\bibitem{benedikter_meanfield_2016}
N.~Benedikter, V.~Jaksic, M.~Porta, C.~Saffirio, and B.~Schlein.
\newblock Mean-{{Field Evolution}} of {{Fermionic Mixed States}}.
\newblock {\em Communications on Pure and Applied Mathematics},
  69(12):2250--2303, 2016.

\bibitem{benedikter_hartree_2016}
N.~Benedikter, M.~Porta, C.~Saffirio, and B.~Schlein.
\newblock From the {{Hartree Dynamics}} to the {{Vlasov Equation}}.
\newblock {\em Archive for Rational Mechanics and Analysis}, 221(1):273--334,
  2016.

\bibitem{bergh_interpolation_1976}
J.~Bergh and J.~L{\"o}fstr{\"o}m.
\newblock {\em Interpolation Spaces. {{An}} Introduction}, volume 223 of {\em
  Grundlehren Der {{Mathematischen Wissenschaften}}}.
\newblock {Springer Berlin Heidelberg}, {Berlin, Heidelberg}, 1976.

\bibitem{castella_solutions_1997}
F.~Castella.
\newblock {$L^2$} solutions to the {{Schr\"odinger}}\textendash{{Poisson
  System}}: {{Existence}}, {{Uniqueness}}, {{Time Behaviour}}, and {{Smoothing
  Effects}}.
\newblock {\em Mathematical Models and Methods in Applied Sciences},
  7(08):1051--1083, 1997.

\bibitem{chen_combined_2021}
L.~Chen, J.~Lee, and M.~Liew.
\newblock Combined {{Mean-Field}} and {{Semiclassical Limits}} of {{Large
  Fermionic Systems}}.
\newblock {\em Journal of Statistical Physics}, 182(2):24, Jan. 2021.

\bibitem{chong_manybody_2021}
J.~J.~W. Chong, L.~Lafleche, and C.~Saffirio.
\newblock From {{Many-Body Quantum Dynamics}} to the
  {{Hartree}}\textendash{{Fock}} and {{Vlasov Equations}} with {{Singular
  Potentials}}.
\newblock {\em arXiv:2103.10946}, pages 1--74, Mar. 2021.

\bibitem{chong_globalintime_2022}
J.~J.~W. Chong, L.~Lafleche, and C.~Saffirio.
\newblock Global-in-time {{Semiclassical Regularity}} for the
  {{Hartree}}\textendash{{Fock Equation}}.
\newblock {\em arXiv:2202.13998}, pages 1--11, Feb. 2022.

\bibitem{figalli_semiclassical_2012}
A.~Figalli, M.~Ligab{\`o}, and T.~Paul.
\newblock Semiclassical {{Limit}} for {{Mixed States}} with {{Singular}} and
  {{Rough Potentials}}.
\newblock {\em Indiana University Mathematics Journal}, 61(1):193--222, 2012.

\bibitem{golse_schrodinger_2017}
F.~Golse and T.~Paul.
\newblock The {{Schr\"odinger Equation}} in the {{Mean-Field}} and
  {{Semiclassical Regime}}.
\newblock {\em Archive for Rational Mechanics and Analysis}, 223(1):57--94,
  2017.

\bibitem{golse_meanfield_2021}
F.~Golse and T.~Paul.
\newblock Mean-{{Field}} and {{Classical Limit}} for the {$N$}-{{Body Quantum
  Dynamics}} with {{Coulomb Interaction}}.
\newblock {\em Communications on Pure and Applied Mathematics}, pages 1--35,
  Mar. 2021.

\bibitem{graffi_meanfield_2003}
S.~Graffi, A.~Martinez, and M.~Pulvirenti.
\newblock Mean-{{Field Approximation}} of {{Quantum Systems}} and {{Classical
  Limit}}.
\newblock {\em Mathematical Models and Methods in Applied Sciences},
  13(01):59--73, Jan. 2003.

\bibitem{lafleche_propagation_2019}
L.~Lafleche.
\newblock Propagation of {{Moments}} and {{Semiclassical Limit}} from
  {{Hartree}} to {{Vlasov Equation}}.
\newblock {\em Journal of Statistical Physics}, 177(1):20--60, Oct. 2019.

\bibitem{lafleche_global_2021}
L.~Lafleche.
\newblock Global {{Semiclassical Limit}} from {{Hartree}} to {{Vlasov
  Equation}} for {{Concentrated Initial Data}}.
\newblock {\em Annales de l'Institut Henri Poincar\'e C, Analyse non
  lin\'eaire}, 38(6):1739--1762, Nov. 2021.

\bibitem{lafleche_strong_2021}
L.~Lafleche and C.~Saffirio.
\newblock Strong {{Semiclassical Limit}} from {{Hartree}} and
  {{Hartree}}\textendash{{Fock}} to {{Vlasov}}\textendash{{Poisson Equation}}.
\newblock {\em Analysis {$\&$} PDE}, accepted for publication:1--35, Oct. 2021.

\bibitem{lerner_fefferman-phong_2007}
N.~Lerner and Y.~Morimoto.
\newblock On the {{Fefferman-Phong}} inequality and a {{Wiener-type}} algebra
  of pseudodifferential operators.
\newblock {\em Publications of the Research Institute for Mathematical
  Sciences}, 43(2):329--371, 2007.

\bibitem{lewin_hartree_2020}
M.~Lewin and J.~Sabin.
\newblock The {{Hartree}} and {{Vlasov}} equations at positive density.
\newblock {\em Communications in Partial Differential Equations},
  45(12):1702--1754, Dec. 2020.

\bibitem{lions_mesures_1993}
P.-L. Lions and T.~Paul.
\newblock Sur les mesures de {{Wigner}}.
\newblock {\em Revista Matem\'atica Iberoamericana}, 9(3):553--618, 1993.

\bibitem{lions_propagation_1991}
P.-L. Lions and B.~Perthame.
\newblock Propagation of moments and regularity for the 3-dimensional
  {{Vlasov}}\textendash{{Poisson}} system.
\newblock {\em Inventiones Mathematicae}, 105(2):415--430, 1991.

\bibitem{markowich_classical_1993}
P.~A. Markowich and N.~J. Mauser.
\newblock The {{Classical Limit}} of a {{Self-Consistent Quantum Vlasov
  Equation}}.
\newblock {\em Mathematical Models and Methods in Applied Sciences},
  3(01):109--124, Feb. 1993.

\bibitem{narnhofer_vlasov_1981}
H.~Narnhofer and G.~L. Sewell.
\newblock Vlasov hydrodynamics of a quantum mechanical model.
\newblock {\em Communications in Mathematical Physics}, 79(1):9--24, Mar. 1981.

\bibitem{pfaffelmoser_global_1992}
K.~Pfaffelmoser.
\newblock Global classical solutions of the {{Vlasov}}\textendash{{Poisson}}
  system in three dimensions for general initial data.
\newblock {\em Journal of Differential Equations}, 95(2):281--303, Feb. 1992.

\bibitem{powers_free_1970}
R.~T. Powers and E.~St{\o}rmer.
\newblock Free {{States}} of the {{Canonical Anticommutation Relations}}.
\newblock {\em Communications in Mathematical Physics}, 16(1):1--33, 1970.

\bibitem{saffirio_semiclassical_2019}
C.~Saffirio.
\newblock Semiclassical {{Limit}} to the {{Vlasov Equation}} with {{Inverse
  Power Law Potentials}}.
\newblock {\em Communications in Mathematical Physics}, 373(2):571--619, Mar.
  2019.

\bibitem{saffirio_hartree_2020}
C.~Saffirio.
\newblock From the {{Hartree}} equation to the {{Vlasov}}\textendash{{Poisson}}
  system: Strong convergence for a class of mixed states.
\newblock {\em SIAM Journal on Mathematical Analysis}, 52(6):5533--5553, Jan.
  2020.

\bibitem{spohn_vlasov_1981}
H.~Spohn.
\newblock On the {{Vlasov}} hierarchy.
\newblock {\em Mathematical Methods in the Applied Sciences}, 3(1):445--455,
  1981.

\end{thebibliography}

\end{document}